\documentclass[12pt]{article}
\usepackage[amsmath]{e-jc}
\usepackage{biblatex} 
\newcommand{\Z}{\ensuremath{\mathbb{Z}}}

\newcommand{\C}{\ensuremath{\mathbb{C}}}
\newcommand{\VE}{\ensuremath{\mathcal{V}}}
\newcommand{\N}{\ensuremath{\mathbb{N}}}
\newcommand{\R}{\ensuremath{\mathbb{R}}}
\newcommand{\Q}{\ensuremath{\mathbb{Q}}}
\newcommand{\bigslant}[2]{{\raisebox{.2em}{$#1$}\left/\raisebox{-.2em}{$#2$}\right.}}
\newcommand{\lap}{\ensuremath{\mathcal{L}}}
\usepackage{mathrsfs}
\usepackage{mathtools}
\usepackage{enumitem} 
\usepackage{graphicx}
\usepackage{subcaption}
\usepackage{etoolbox}
\apptocmd{\sloppy}{\hbadness 10000\relax}{}{}
\numberwithin{equation}{section}

\usepackage{graphicx}

\dateline{2.11.2022}{TBC}{TBD}

\MSC{31A30, 05C81, 60J65, 30E25, 05A15}

\Copyright{The authors. Released under the CC BY-ND license (International 4.0).}

\title{Polyharmonic Functions in the Quarter Plane}

\author{Andreas Nessmann\thanks{This project has received funding from the European Research Council (ERC) under the European Union's Horizon 2020 research and innovation programme under the Grant Agreement No. 759702.}\\
\small Institut für Diskrete Mathematik und Geometrie,\\[-0.8ex] 
\small Technische Universität Wien, Vienna, Austria\\
\small Institut Denis Poisson,\\[-0.8ex] 
\small Université de Tours, Tours, France\\
\small\tt andreas.nessmann@tuwien.ac.at}
\begin{document}

\maketitle


\begin{abstract}

  While discrete harmonic functions have been objects of interest for quite some time, this is not the case for discrete polyharmonic functions, as appear for instance in the asymptotics of path counting problems. In this article, a novel method to compute all discrete polyharmonic functions in the quarter plane for non-singular models with small steps and zero drift is proposed. In case of a finite group, an alternative method using decoupling functions is given, which often leads to a basis consisting of rational functions. In a similar manner one can obtain polyharmonic functions in the continuous setting, and convergence between the discrete and continuous cases is proven. Lastly, using a concrete example it is shown why the decoupling approach seems not to work in the infinite group case.
\end{abstract}

  I would like to thank Michael Drmota and Kilian Raschel for introducing me to this topic, and especially the latter for many helpful and enlightening discussions.
  
\section{Introduction and Motivation}
Suppose we are given a weighted step set $\mathcal{S}\subseteq \{-1,0,1\}^2$, and we want to count the (weighted) number $q(0,x;n)$ of excursions in the quarter plane $\Z_{\geq 0}\times\Z_{\geq 0}$ of length $n$ from the origin to some point $x=(i,j)$. For the simple walk for instance, we have $\mathcal{S}=\left\{\uparrow,\rightarrow,\downarrow,\leftarrow\right\}$, where each step has weight $\frac{1}{4}$. In this case, the number $q(0,x;n)$ can be computed explicitly (see e.g. \cite{MBM2}) via \begin{align}\label{eq:SW_exact}
    q(0,x;n)=4^{-n}\frac{(i+1)(j+1)n!(n+2)!}{m!(m+i+1)!(m+j+1)!(m+i+j+2)!},
\end{align}
where $m=\frac{n-i-j}{2}$ is an integer, and $0$ otherwise (as the quarter plane is bipartite). It is now fairly natural to ask about asymptotics of this expression, or more generally about asymptotics of the number $q(0,x;n)$ for an arbitrary step set $\mathcal{S}$. In particular, we consider (as proposed in \cite{Poly}) asymptotic expansions such that there is a (strictly) increasing sequence $(\alpha_p)$ such that for all $p$ we have (again, up to some coefficients possibly vanishing due to parity considerations) \begin{align}\label{eq:asymptotic}
q(0,x;n)=\gamma^n\left(\sum_{p=1}^k\frac{v_p(x)}{n^{\alpha_p}}+\mathcal{O}\left(\frac{1}{n^{\alpha_{k+1}}}\right)\right).
\end{align}
In case of the simple walk, (\ref{eq:SW_exact}) allows us to directly compute $\alpha_p=p+2$ and \begin{align}
    \label{eq:as_SW_1}v_1=&(i+1)(j+1),\\
    v_2=&(i+1)(j+1)(15 + 4 i + 2 i^2 + 4 j + 2 j^2),\\
    \label{eq:as_SW_2}
    \begin{split}v_3=&(i + 1) (j + 1) (317 + 16 i^3 + 4 i^4 + 168 j + 100 j^2 + 16 j^3 \\
    &   +4j^4+ 8 i (21 + 4 j + 2 j^2) + 4 i^2 (25 + 4 j + 2 j^2)).
   \end{split}
\end{align} 
It should be explicitly noted at this point that expansions of the form (\ref{eq:asymptotic}) are not proven to exist for this type of problem. While for the simple walk and a few other examples (e.g.~the diagonal walk, tandem walk, see \cite{MBM}) this can be shown using an explicit representation similar as (\ref{eq:SW_exact}), in general it is not so clear. One-term expansions of this form have been proven for many cases in \cite{Denisov}, and more recently, using multivariate analytic techniques, in \cite[Thm.~1]{Universality}, \cite[6.1]{ACSV}, while higher order asymptotics for the one-dimensional case have been shown in \cite{WachtelPhD}.\\
It is now fairly natural to ask how the asymptotics depend on the finishing point of our paths, which is to ask about the properties of the $v_p$: whether they necessarily have a particular structure, if there is a clear relation to our chosen step set, and how to compute them. It is a very recent observation from the extended abstract \cite{Poly} that each $v_p$ must be a polyharmonic function of order $p$, which to a large extent answers the first two questions. This can be shown by utilizing a recursive relation between the $q(0,x;n+1)$ and $q(0,x;n)$, and showing that each function $v_p$ must be what is called a discrete polyharmonic function of order $p$. This article aims to take a closer look at the structure of these functions and give a method to construct them, which to the author's knowledge has not been done before, albeit in \cite{Poly} the authors were able to compute some biharmonic functions using a guessing approach which will be more closely examined in Section~\ref{sec:guessing}.\\
In the continuous case, given a covariance matrix $\Sigma = \begin{pmatrix} \sigma_{11} & \sigma_{12}\\ \sigma_{12} & \sigma_{22}\end{pmatrix}$, we call a function $f$ polyharmonic of degree $p$ if it is a solution of \begin{align}\label{eq:defPHF}
    \triangle^p f=0,
\end{align}
where $\triangle$ is the Laplace-Beltrami operator $\triangle = \frac{1}{2}\left(\sigma_{11}\frac{\partial^2}{\partial x^2}+2\sigma_{12}\frac{\partial ^2}{\partial x\partial y}+\sigma_{22}\frac{\partial^2}{\partial y^2}\right)$. These kinds of functions have already been studied in the late 19th century, notably by E.~Almansi, who proved in \cite{Almansi} that in a star-shaped domain containing the origin, any polyharmonic function of degree $n$ can be written as \begin{align}
    f(x)=\sum_{k=0}^n |x|^{2k}h_k(x),
\end{align} 
where the $h_k$ are harmonic (polyharmonic of degree $1$). In particular harmonic and biharmonic functions have by now seen plenty of applications in physics, see e.g. \cite{Elasticity}.\\
The discrete setting on the other hand has gained interest comparably recently. A (discrete) function defined on a graph is called polyharmonic if it satisfies (\ref{eq:defPHF}) as well, but with a discretised version of the Laplacian. For this discretisation, given transition probabilities $p_{x,y}$ from any point $x$ to any point $y$, one lets \begin{align}
    \triangle f(x)=\sum_{y}p_{x,y}f(y)-f(x).
\end{align}
There have been some results on polyharmonic functions on trees recently \cite{Trees2,Trees1}, and polyharmonic functions on subdomains of $\mathbb{Z}^d$ have become an object of interest linked in particular to the study of discrete random walks.  In our case, this subdomain will be the quarter plane and our walk homogeneous inside of it, i.e.~the transition probabilities $p_s:=p_{x,x+s}$, where the steps $s$ are given by the set $\mathcal{S}$ of allowed steps, will be independent of $x$. The discrete Laplacian thus reads \begin{align}\label{eq:intro_1}
    \triangle f(i,j)=\sum_{(u,v)\in\mathcal{S}}p_{u,v}f(i+u,j+v)-f(i,j).
\end{align}
One can immediately verify that the functions given by (\ref{eq:as_SW_1})-(\ref{eq:as_SW_2}) are indeed polyharmonic of degrees $1,2,3$ respectively. It is not at all obvious, however, how polyharmonic functions in general can be found. In \cite{Conformal}, a way to construct harmonic functions for zero-drift models with small steps via a boundary value problem is given. This is utilized in \cite{Hung} to give a complete description of harmonic functions for symmetric step sets with small negative steps, which has since been extended to results for the non-symmetric case in \cite{Hung2}. The methods used in the latter two articles can be applied to compute harmonic functions in the setting considered here (i.e. small steps, zero drift, non-degenerate models; see Section~\ref{sec:DiscretePHF}) with only minor adjustments. Very recently in the extended abstract \cite{Poly}, the authors outline a way to compute biharmonic functions. They utilize a guessing approach, which works -- with some restrictions -- in the finite group setting, which will also be discussed in Section~\ref{sec:guessing}. Their main idea in doing so is to find a so-called `decoupling function', which was first introduced by W.~T.~Tutte in \cite{Tutte2}, and is discussed further in \cite{Tutte}. There, this concept is utilized to give remarkably succinct proofs of the algebraicity (or D-algebraicity) of the counting function of some models in the quarter plane. While in \cite{Poly} the authors guessed decoupling functions for some concrete examples, it is possible to show for finite group models that they always exist, and describe them explicitly. This will be done in Section~\ref{sec:Decoupling}.\\
What all these articles have in common and will be the same here is that instead of working directly with a polyharmonic function $h(i,j)$, they consider its generating function $H(x,y):=\sum_{i,j}x^{i+1}y^{j+1}h(i,j)$. The main reason to do so is the functional equation \begin{small}
\begin{align}
    \label{eq:FE} K(x,y)H(x,y)&=K(x,0)H(x,0)+K(0,y)H(0,y)-K(0,0)H(0,0)-xy\left[\triangle H\right](x,y),
\end{align}
\end{small}\noindent
which can be shown by straightforward computation to be satisfied by this generating function (note that we have $\triangle H=0$ for harmonic $H$). Here, $K(x,y)$, which will be defined in Section~\ref{sec:Prelims}, is a the same kernel that usually appears in the study of random walks except with the directions reversed. Note that the functional equations for counting walks or the stationary distribution look strikingly similar, see e.g. \cite{Book,MBM}.
\medbreak
The goal of this article is to give an overview of the structure and an algorithm to compute all discrete polyharmonic functions for walks with small steps and zero drift in the quarter plane. If the group is finite, then one can expand on the idea in \cite{Poly} and use decoupling functions, resulting in many cases in a particularly nice basis consisting of rational functions of a fairly simple shape. In this case, one also obtains a direct link to continuous polyharmonic functions. \newline
The structure of this article will be roughly as follows:
\begin{itemize}
\item In Section~\ref{sec:Prelims}, a short introduction to the setting as well as a quick overview of some tools that will be utilized is given.
\item In Section~\ref{sec:DiscretePHF}, some general properties of discrete polyharmonic functions will be stated.
\item In Section~\ref{sec:generalSolution}, a general algorithm to construct discrete polyharmonic functions is presented (Thm.~\ref{thm:ConstructPHF}), and it is shown that all possible discrete polyharmonic functions can be constructed in this manner (Thm.~\ref{thm:AllPHF_generalcase}).
\item In Section~\ref{sec:continuousPHF}, an analogue of the latter method in the continuous case is presented, the relation between the discrete and continuous functional equations as well as convergence in terms of generating functions and Laplace transforms are discussed. 
\item In Section~\ref{sec:Decoupling}, an alternative construction utilizing decoupling functions is presented, which is applicable to models with finite group only (Thm.~\ref{thm:buildPHF}). This method leads, provided a certain parameter is integer, to a (Schauder) basis consisting of rational functions. This construction is then translated to the continuous setting, and convergence properties are shown (Thm.~\ref{thm:Convergence}).
\item In Section~\ref{sec:guessing}, the guessing approach mentioned (but not detailed) in \cite{Poly} is discussed, in which one uses an ansatz to try and find suitable decoupling functions. In particular, we deduce that the guessing approach works to decide if there is a decoupling function of a reasonably nice shape.
\item In Section~\ref{sec:pitheta2}, a special case is examined in more detail, to show the relation between the two approaches and in particular why it does not appear promising to extend the notion of decoupling to the infinite group case. 
\item Lastly, Section~\ref{sec:Outlook} gives a brief overview of some open questions.
\end{itemize}
This article is the complete version of the extended abstract \cite{AofA}.

\section{Preliminaries}\label{sec:Prelims}
The following only serves as a very brief overview; for a more thorough introduction see e.g. \cite{Conformal, Book}. Consider a homogeneous\footnote{That is, the probability to jump from a point $x$ to a point $y$ depends on $y-x$ only.} random walk in $\mathbb{Z}\times\mathbb{Z}$ with a step set $\mathcal{S}$ and transition probabilities $p_{i,j}$. From now on, we will make the following assumptions:
\begin{enumerate}[label=(\roman*)]
    \item The walk consists of small steps only, i.e. $\mathcal{S}\subseteq\{-1,0,1\}^2$.
    \item The walk is non-degenerate, that is, the list $p_{1,1}$,$p_{1,0}$,$p_{1,-1}$,$p_{0,-1}$,$p_{-1,-1}$,$p_{-1,0}$,$p_{-1,1}$,$p_{0,1}$ does not contain three consecutive $0$s. 
    \item The walk has zero drift, meaning that $\sum_{(i,j)\in\mathcal{S}}ip_{i,j}=\sum_{(i,j)\in\mathcal{S}}jp_{i,j}=0.$
\end{enumerate}
A standard object appearing in a variety of functional equations around random walks (for example when one wants to compute a stationary distribution, or for counting walks, see e.g. \cite{Book}) is the \textit{kernel} of the walk, which is given by \begin{align}
    K(x,y)=xy\left(\sum_{(i,j)\in\mathcal{S}}p_{i,j}x^{-i}y^{-j}-1\right).
\end{align}

Note that this kernel slightly differs from the one used for counting walks: on top of the absence of the counting variable $t$, the directions of the steps is reversed (i.e.~the monomial $x^iy^j$ is paired with the probability $p_{-i,-j}$ here instead of the other way around, for instance in~\cite{MBM}). This can be intuitively explained by the fact that when counting walks they are commonly grouped by their previous step, whereas when looking at asymptotics we build our recursions by looking forward. In \cite{Book}, the kernel is examined quite thoroughly, and we will in the following state a few of their results. Note that they are not affected by reversing the directions of our steps.\\
As we consider non-degenerate walks with small steps, our kernel will necessarily be quadratic in both $x$ and $y$. Letting \begin{align}
\label{eq:kernelnot1}
    K(x,y)=a(x)y^2+b(x)y+c(x)= \tilde{a}(y)x^2+\tilde{b}(y)x+\tilde{c}(y),
\end{align}
we can use the quadratic formula to find solutions of $K(\cdot,y)=0$ and $K(x,\cdot)=0$, which are given by \begin{align}
    X_\pm (y)=\frac{-\tilde{b}(y)\pm\sqrt{\tilde{b}(y)^2-4\tilde{a}(y)\tilde{c}(y)}}{2\tilde{a}(y)},\quad
    Y_\pm(x)=\frac{-b(x)\pm\sqrt{b(x)^2-4a(x)c(x)}}{2a(x)}.
\end{align}
Letting $D(x):=b(x)^2-4a(x)c(x)$, then one can show \cite[2.5]{Conformal}, \cite[2.3.2]{Book} that given our particular case of zero drift, $D(x)=0$ has $3$ solutions: the double root $x=1$, a solution $x_1\in [-1,1)$, and a solution $x_4\in (1,\infty)\cup (-\infty, -1]$. Consequently, one can see that the discriminant is negative for $y\in[y_1,1]$, and therefore in this range we have $X_+(y)=\overline{X_-(y)}$. Analogous results hold for $\tilde{D}(y):=\tilde{b}(y)^2-4\tilde{a}(y)\tilde{c}(y)$. This is in particular used in the computation of harmonic functions, as in \cite{Conformal} or \cite{Hung}. The idea is to define the domain $\mathcal{G}$ as the area bounded by the curve $X_\pm\left([y_1,1]\right)$, and notice that the functional equation~(\ref{eq:FE}) leads to the boundary value problem \begin{align}
    K(x,0)H(x,0)-K(\overline{x},0)H(\overline{x},0)=0
\end{align}
on $\partial\mathcal{G}\setminus\{1\}$, while $K(x,0)H(x,0)$ is analytic in the interior of $\mathcal{G}$ and continuous on $\overline{\mathcal{G}}\setminus\{1\}$ (cf \cite{Poly,Conformal}). A few examples of what $\mathcal{G}$ can look like is given in Fig.~\ref{fig:gammas}; in particular the case where $\mathcal{G}$ is the unit disk will be examined in Section~\ref{sec:pitheta2}. 
\begin{figure}
\centering
\begin{subfigure}{.5\linewidth}
  \centering
  \includegraphics[width=.3\linewidth]{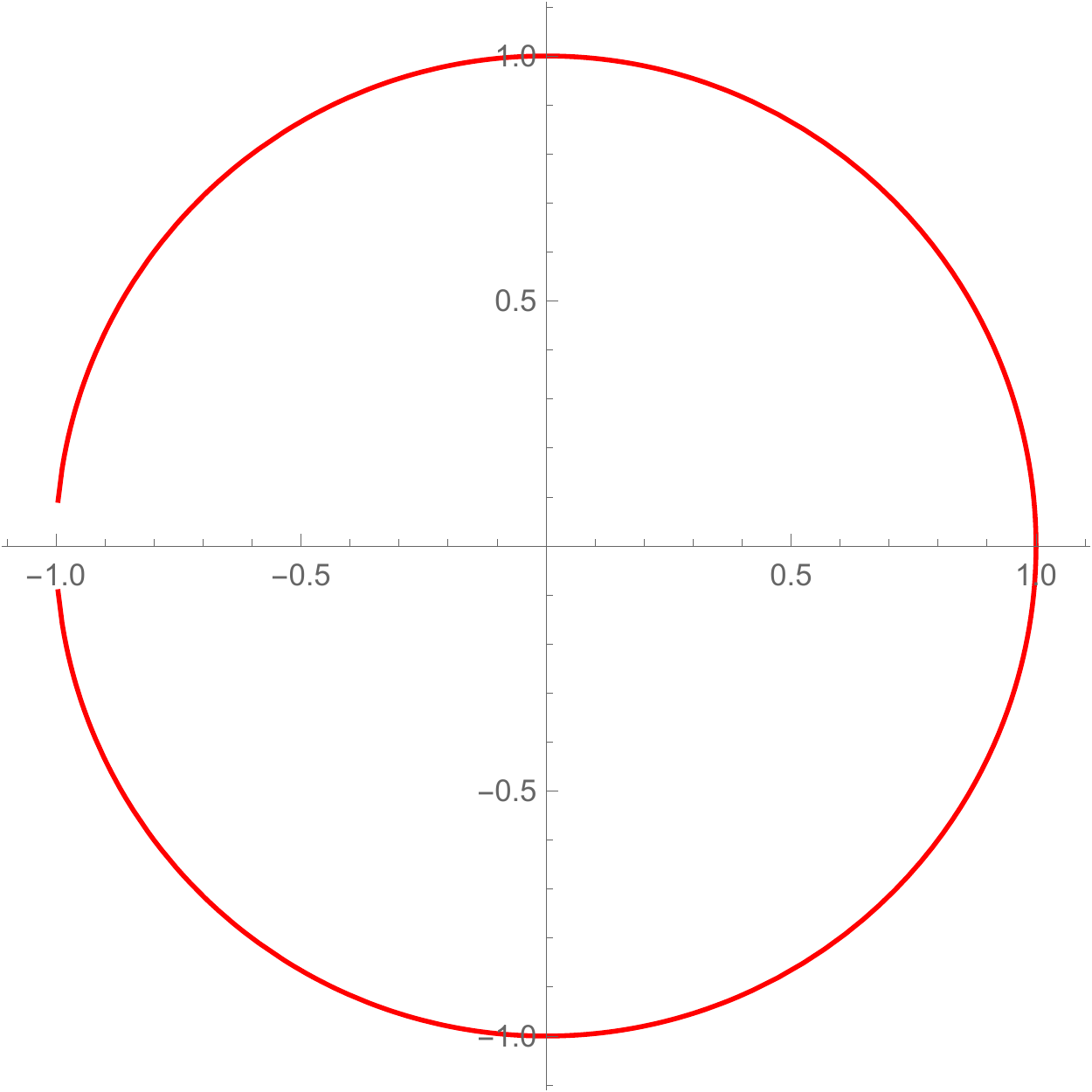}
  \caption{The simple walk: $\mathcal{S}=\{\rightarrow,\downarrow,\leftarrow,\uparrow\}$}
  \label{fig:sub1}
\end{subfigure}%
\begin{subfigure}{.25\linewidth}
  \centering
  \includegraphics[width=.5\linewidth]{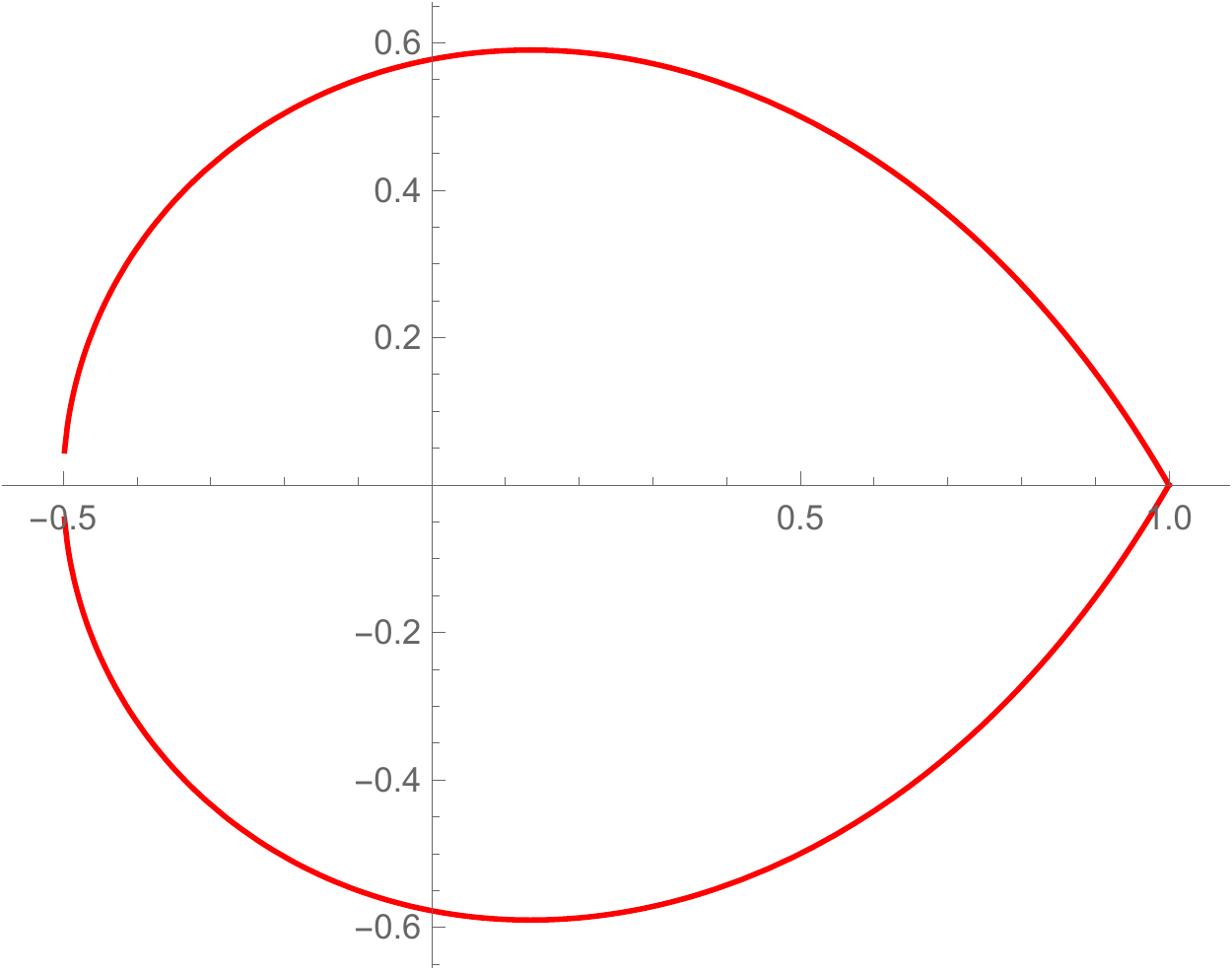}
  \caption{The tandem walk: $\mathcal{S}=\{\nwarrow,\rightarrow,\downarrow\}$}
  \label{fig:sub2}
\end{subfigure}
\begin{subfigure}{.4\textwidth}
  \centering
  \includegraphics[width=.4\linewidth]{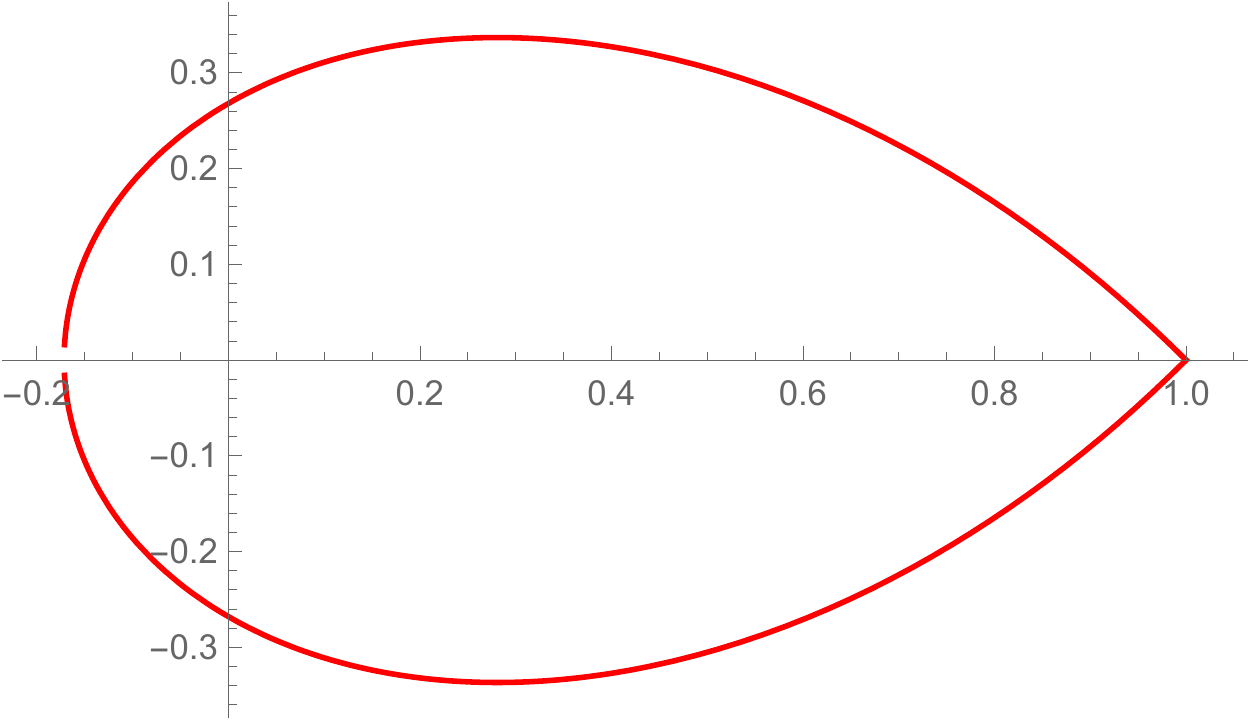}
  \caption{The Gouyou-Beauchamps walk: $\mathcal{S}=\{\leftarrow,\rightarrow,\nwarrow,\searrow\}$}
  \label{fig:sub3}
\end{subfigure}
\caption{The shape of the domain $\mathcal{G}$ for different (unweighted) models. The angle $\theta$ can be defined as the angle at which its boundary intersecs the real axis at the point $(1,1)$.}
\label{fig:gammas}
\end{figure}

In order to solve the above boundary value problem, one can construct a mapping $\omega: \mathbb{C}\to \bar{\mathbb{C}}$ which is a fundamental solution in the sense that any other solution can be written as some entire function applied to $\omega$. An explicit formula for $\omega$ as well as some additional informations are given in \cite{Conformal}. This $\omega$ satisfies \begin{align}\label{eq:wInv}
    \omega(0)&=0,\quad
    \omega(X_+(y))=\omega(X_-(y))\quad\forall y\in[y_1,1],\quad 
   \frac{\partial \omega}{\partial x}(x)\neq 0\quad\forall x\in\mathcal{G}^{\circ}.
\end{align}
In particular, $\omega$ is a conformal mapping of the domain $\mathcal{G}$. Furthermore, it has a pole-like singularity of order $\pi/\theta$ at $x=1$, where $\theta$ is the inner angle at which $\partial\mathcal{G}$ intersects the $x$-axis. Alternatively, $\theta$ can be computed via \begin{align}\label{eq:defTheta}
    \theta=\arccos\left(-\frac{\sum ij p_{i,j}}{\sqrt{\sum i^2 p_{i,j}}\sqrt{\sum j^2 p_{i,j}}}\right),
\end{align} see e.g.~\cite[2.15]{Conformal}. This angle $\theta$ happens to be closely related to the so-called group of the walk; whenever the group is finite, the ratio $\pi/\theta$ must be rational. This group will be defined in Section~\ref{sec:Decoupling} and is by now a standard object of interest in the combinatorics of lattice paths (see e.g.~\cite{MBM,MBM2}). Also, the angle $\theta$ is directly linked to the asymptotic growth of harmonic functions, see e.g.~\cite{Hung,Hung2}.\\
In the same manner as one has constructed the region $\mathcal{G}$, one can obtain a $\mathcal{G}'$ by swapping the roles of $x$ and $y$. While it is possible to construct a second conformal mapping $\widehat{\omega}$ for $\mathcal{G}'$ in the same fashion as $\omega$, one can also see that $\omega\circ X_+$ is a conformal mapping by \cite[Cor.~5.3.5]{Book}, and it has the same behaviour around $1$ as $\omega$. Finally, we note that due to (\ref{eq:wInv}), $\omega$ is an invariant in the sense of \cite[Def.~4.3]{Tutte}.

\section{Discrete Polyharmonic Functions}\label{sec:DiscretePHF}
Let $\VE$ be the quarter plane $\Z_{\geq 0}\times\Z_{\geq 0}$, with $\VE^\circ=\{(x_1,x_2)\in\VE: x_1,x_2\neq 0\}$ and $\partial\VE=\VE\setminus\VE^\circ$. For any function $f: \VE\to\C$, we define \begin{align}\label{eq:dfn_Laplacian}
\triangle f(x):=\sum_{s\in\mathcal{S}}p_s f(x+s) -f(x),\end{align}
where $x=(x_1,x_2)\in\VE$ and the sum $x+s$ is to be taken component-wise. 

\begin{definition}\label{def:Discrete_PHF}A function $f: \mathbb{Z}_{\geq 0}\times\mathbb{Z}_{\geq 0}\to \mathbb{C}$ is called \textbf{polyharmonic of degree $k$}, if \begin{align}
    \triangle^k f(x)&=0\quad\forall x\in\VE^\circ,\\
    f(x)&=0\quad\forall x\in\partial\VE.
\end{align}
We call $f$ \textbf{harmonic} if it is polyharmonic of degree $1$. 
\end{definition}
The Dirichlet condition is a direct consequence of the probabilistic interpretation given in the introduction; clearly there are no walks starting outside of the quarter plane which are always inside it. First of all, we will see that given combinatorial problems similar as the path counting one mentioned in the introduction, it is reasonable to expect polyharmonic functions to appear in the asymptotics. 
\begin{lemma}\label{lemma:asymptotics_PHF}
Let $q(x;n)$ be some combinatorial quantity depending on $n$ and a point $x\in\Z^d$, and let $\mathcal{S}\subset\Z^d$ be some step set. Furthermore, suppose that $q(x;n)$ has an asymptotic expansion of the form \begin{align}
    q(x;n)=\sum_{k=1}^\infty f_k(n)v_k(x),\\
    \end{align}
    where $\lim_{n\to\infty} \frac{f_{k+1}(n)}{f_k(n)}=0$ and $\lim_{n\to\infty}\frac{f_k(n+1)}{f_k(n)}=1$ for all $k\geq 1$,
    and that it satisfies a recursive relation of the form
    \begin{align}
    q(x;n+1)=\sum_{s\in\mathcal{S}}p_s q(x-s;n).
\end{align}
Then, for all $k\geq 1$, $v_k(x)$ is a polyharmonic function of degree $k$.
\end{lemma}
\begin{proof}
Suppose we already know that for $p=1,\dots, m$, $v_p$ is $p$-polyharmonic. We have \begin{align}
    \triangle^m q(x;n+1)&=\sum_{s\in\mathcal{S}}p_s\triangle^mq(x-s;n)\quad\Leftrightarrow \\
    \sum_{k=1}^\infty f_k(n+1)\triangle^m v_k(x)&=\sum_{s\in\mathcal{S}}\sum_{k=1}^\infty p_sf_k(n)\triangle^mv_k(x)\quad\Leftrightarrow\\
    \sum_{k=m+1}^\infty f_k(n+1)\triangle^m v_k(x)&=\sum_{s\in\mathcal{S}}\sum_{k=m+1}^\infty p_sf_k(n)\triangle^mv_k(x)\quad\Rightarrow\\
    \sum_{k=m+1}^\infty \frac{f_k(n+1)}{f_{m+1}(n)}\triangle^m v_k(x)&=\sum_{s\in\mathcal{S}}\sum_{k=m+1}^\infty p_s\frac{f_k(n)}{f_{m+1}(n)}\triangle^m v_k(x-s)\quad\Rightarrow\\
    \triangle^m v_{m+1}(x)&=\sum_{s\in\mathcal{S}}p_s \triangle^m v_{m+1}(x-s),
\end{align}
where in the last line we take the limit $n\to\infty$. This tells us that $\triangle^mv_{m+1}(x)$ is harmonic, which the same as $v_{m+1}(x)$ being $m+1$-polyharmonic.
\end{proof}
\textbf{Remark:} While this shows us that $v_k(x)$ is $k$-polyharmonic, it does $\textit{not}$ imply that $\triangle v_{k+1}(x)=v_k(x)$. Also, the above lemma covers in particular the case where the $f_k(n)$ are rational functions with decreasing degrees, e.g. $f_k(n)=n^{c-k}$.
\vskip\baselineskip
Denote in the following by $\mathcal{H}_n$ the space of discrete $n$-polyharmonic functions as defined in Def.~\ref{def:Discrete_PHF}, and by $\mathcal{H}:=\bigcup_{n\in\N}\mathcal{H}_n$ the space of all discrete polyharmonic functions. Clearly, $\mathcal{H}_n$ is a $\mathbb{C}$-vector space. Now, given any $\widehat{H}_n\in\mathcal{H}_n$, we can identify it with the sequence $\left(\widehat{H}_n,\widehat{H}_{n-1},\dots, \widehat{H}_1\right)$, where $\triangle\widehat{H}_{k+1}=\widehat{H}_k$, and $\triangle\widehat{H}_1=0$. It is clear that any such sequence is uniquely defined by the corresponding $\widehat{H}_n$. Now suppose that we have $\widehat{H}_n,\widehat{H}_n'\in \mathcal{H}_n$, such that, with their sequence representation as above, $\widehat{H}_1=\widehat{H}_1'$. In this case, we have \begin{align}
    \triangle^{n-1}\left[\widehat{H}_n-\widehat{H}_n'\right]=\widehat{H}_1-\widehat{H}_1'=0,
\end{align}
thus $\widehat{H}_n-\widehat{H}_n'\in\mathcal{H}_{n-1}$. Therefore, provided that for each $\widehat{H}_n\in\mathcal{H}_n$ we can find a corresponding $\widehat{H}_{n+1}\in\mathcal{H}_{n+1}$, which will be shown below in Thms.~\ref{thm:AllHF} and \ref{thm:AllPHF_generalcase}, one can prove the following lemma:
\begin{lemma}\label{lemma:PHFspaces_isomorphism}
Let $\mathcal{H}_n$ be the space of real-valued, discrete $n$-polyharmonic functions in the quarter plane. Then we have an isomorphy of vector spaces \begin{align}
    \mathcal{H}_n\cong \left(\mathcal{H}_1\right)^n.
\end{align}
\end{lemma}
\begin{proof}
Suppose the statement holds for $k=1,\dots,n$ and we have $H_{n+1},H_{n+1}'\in\mathcal{H}_{n+1}$ such that $H_n=H_n'$. Then we have \begin{align}
    \triangle\left[H_{n+1}-H_{n+1}'\right]=H_n-H_n'=0,
\end{align}
thus $H_{n+1}-H_{n+1}'\in\mathcal{H}_1$. Therefore (and utilizing in advance Thm.~\ref{thm:AllPHF_generalcase}), we can construct an isomorphism $\bigslant{\mathcal{H}_{n+1}}{\mathcal{H}_1}\to\mathcal{H}_n$, and the proof is complete.
\end{proof}
In particular, if we are given any $\widehat{H}_n\in\mathcal{H}_n$, and we want to find all $\widehat{H}_{n+1}\in\mathcal{H}_{n+1}$ with $\triangle \widehat{H}_{n+1}=\widehat{H}_n$, then this means that it suffices to find a single $\widehat{H}_{n+1}$ with this property as well as all harmonic functions, because any other such $\widehat{H}_{n+1}'$ can be written as $\widehat{H}_{n+1}+\widehat{G}_1$, for some $\widehat{G}_1\in\mathcal{H}_1$.\\
We already know (see e.g. \cite{Conformal},\cite{Poly}, using the idea of the BVP outlined above), that for any entire function\footnote{Note that the inverse does not hold, i.e. not any harmonic function can be written in this manner; as we will see that there are some harmonic functions whose generating functions have radius of convergence $0$.} $P(x)\in\mathbb{R}[x]$, we can construct (the GF of) a harmonic function via 
\begin{align}\label{eq:buildHarmonic}
    H(x,y)=\frac{P\left(\omega(x)\right)-P\left(\omega\left(X_+(y)\right)\right)}{K(x,y)},
\end{align}
where $\omega$ is the conformal mapping introduced in Section~\ref{sec:Prelims}. We will now show what is, in a sense, the opposite direction of the above statement. The following theorem (as well as its proof) is an analogue to \cite[Thm.~2]{Hung}, where a similar result is shown for the case of symmetric walks with small negative steps. 
\begin{theorem}\label{thm:AllHF}
For any discrete harmonic function with generating function $H(x,y)$, there is a unique formal power series $P(x)$ such that (\ref{eq:buildHarmonic}) holds. In particular, we have an isomorphism \begin{align}
    \mathcal{H}_1\cong \mathbb{R}[[x]].
\end{align}
\end{theorem}

\begin{proof}[Proof (outline).] The arguments are mostly the same as in \cite[Thm.~2]{Hung}. From (\ref{eq:FE}), it follows that $K(x,y)H(x,y)$ is already uniquely defined by the (univariate) boundary terms $K(x,0)H(x,0)$ and $K(0,y)H(0,y)$. The idea is to construct, using appropriate power series $P(x)$ in (\ref{eq:buildHarmonic}), a harmonic function for any given possible boundary condition. If $K(0,0)=0$, which is the same as saying that our model does not include a North-East step, then we cannot write $1/K(x,y)$ as a power series, but we can instead choose $X_+(0)$ such that $X_+(0)=0$, i.e. we can substitute $X_+$ into another power series. Therefore, we will consider two cases:
\begin{enumerate}
    \item $K(0,0)=0$:\\
    In this case, substituting $X_+$ for $x$ in (\ref{eq:FE}) gives \begin{align}
        0=K(X_+,0)H(X_+,0)+K(0,y)H(0,y).
    \end{align}
    Utilizing this to substitute for $K(0,y)H(0,y)$ in (\ref{eq:FE}), we obtain \begin{align}
        K(x,y)H(x,y)=\underbrace{K(x,0)H(x,0)}_{=:P(x)}-\underbrace{K(X_+,0)H(X_+,0)}_{=:P(X_+)}.
    \end{align}
    Setting \begin{align}\label{eq:defH1_1}
        H_1^m(x,y)=\frac{\omega(x)^m-\omega(X_+)^m}{K(x,y)},
    \end{align}
    and utilizing that around $0$ we have (after scaling and potentially switching $x,y$) $\omega(x)=\frac{x(1+p(x))}{(1-x)^{\pi/\theta}}$ (see \cite[5.3]{Book}; use that our walk is not singular), we can iteratively compute coefficients $a_k$ such that $\sum a_j\omega(x)^k=P(x)$. To see that at the end we indeed obtain a power series, one can apply the Weierstra{\ss} preparation theorem. 
    \item $K(0,0)\neq 0$:\\
    In this case, the previous approach does not work anymore since substitution of $X_+$ into an arbitrary power series fails. Instead, let now $\omega(x)=\sum x^nc_n,\omega\left(X_+\right)=\sum y^n d_n$. We know that $c_1,d_1\neq 0, c_0=0$ (see \cite[5.3]{Book}, and notice that $p_{-1,-1}\neq 0$).\\
    We can now proceed by defining \begin{align}
        P_{2m}(z)&=z^m(z-d_0)^m,\\
        P_{2m+1}(z)&=z^{m+1}(z-d_0)^m.
    \end{align}
    Letting \begin{align}\label{eq:defH1_2}
        H_1^m(x,y):=\frac{P_m\left(\omega(x)\right)-P_m\left(\omega(X_+)\right)}{K(x,y)},
    \end{align}
    one can check that the monomial with non-zero coefficient with minimal degree in the series representation of $H_1^m(x,y)$ around $0$ occurs for $k=l=m$ for $m$ even, and $k=l+1=m$ otherwise. Note here that $\omega(x),\omega(X_0)$ have non-vanishing derivatives at $0$ as $0\in\mathcal{G}^\circ$, see \cite[5.3]{Poly}.  From there, given arbitrary power series $Q(x),R(y)$ with $Q(0)=R(0)$, one can again iteratively build coefficients $a_n$ such that $\sum a_nP_n(\omega(x))=Q(x)$, $\sum b_nP_n\left(\omega(X_+)\right)=R(y)$. We have thus constructed a harmonic function with boundary terms $Q(x),R(y)$; since these were arbitrary we are done. Note that as $K(0,0)\neq 0$, the division by $K(x,y)$ is not an issue here.

\end{enumerate}

\end{proof}
Thm.~\ref{thm:AllHF} implies that the functions defined by \begin{align}
\label{eq:DefHF}
H_1^m(x,y):=\frac{P_m(\omega(x))-P_m(\omega(X_+))}{K(x,y)},
\end{align}
where \begin{align}
    &P_m(z):=z^m&\qquad\text{ if }K(0,0)=0,\\
   &\begin{drcases}
   P_{2m}(z):=z^m(z-d_0)^m\\
   P_{2m+1}(z):=z^{m+1}(z-d_0)^m
   \end{drcases}&\qquad\text{ if }K(0,0)\neq 0,
\end{align}
and $d_0=\omega(X_+(0))$, as in the proof of Thm.~\ref{thm:AllHF}, form a Schauder basis\footnote{That is, we can express any function not necessarily via finite, but via countable sums.} of $\mathcal{H}_1$. This, combined with the idea from Lemma~\ref{lemma:PHFspaces_isomorphism}, gives us a criterion for a family of polyharmonic functions to be a Schauder basis of $\mathcal{H}$, the space of all polyharmonic functions.
\begin{lemma}\label{lemma:criterion_AllPHF}
Let $\left(H_n^k\right)_{n,k\in\mathbb{N}}$ be a family of discrete polyharmonic functions, such that \begin{enumerate}
    \item $H_1^k(x,y)=\frac{P_k(\omega(x))-P_k(\omega(X_+))}{K(x,y)}$ as in (\ref{eq:DefHF}),
    \item $\triangle H_{n+1}^k=H_n^k$,
\end{enumerate}Then, the $(H_m^l)$, $1\leq m\leq l, 1\leq k$ form a Schauder basis of $\mathcal{H}_n$, that is, given any $H_n\in\mathcal{H}_n$ there are unique $a_{m,l}$, $1\leq i\leq n$, $j\in\mathbb{N}$ such that \begin{align}
    H_n=\sum_{i=m}^n\sum_{l=1}^\infty a_{m,l}H_m^l.
\end{align}

\end{lemma}

\begin{proof}
We do not a priori assume that sums of the form $\sum_{n\geq 0}H_n^k$ converge as formal power series. Therefore, we will first modify our family $(H_n^k)$, forcing this type of convergence. The main issue here is that we do not have any handle on the order of $H_n^k$ at $0$, written in the following as $\operatorname{deg}_0{H_n^k}:=\min\{u+v\mid \text{ the term } x^uy^v\text{ has non-zero coefficient }\}$. However, (\ref{eq:FE}) implies that if $H_{n+1}^k$ has vanishing boundary terms up to a sufficiently high order -- which we can force by adding a suitable harmonic function, as in the proof of Thm.~\ref{thm:AllHF} --, then we have $\operatorname{deg}_0{H_{n+1}^k}>\operatorname{deg}_0{H_n^k}$. We will utilize this in order to construct a family $(\widehat{H}_n^k)$ such that, for each $n$, we have \begin{enumerate}
    \item $\widehat{H_1^k}=H_1^k$,
    \item $\triangle\widehat{H}_{n+1}^k=\widehat{H}_n^k$,
    \item each $\widehat{H}_n^k$ can be written as a countable sum of the $H_m^k$ for $m\leq n$,
    \item $\operatorname{deg}_0\widehat{H}_{n+1}^k$ is at least $\operatorname{deg}_0\widehat{H}_n^k+1\geq\left[\frac{k}{2}\right]+n-1$,
    \item the $\widehat{H_m^k}$ with $m\leq n$ form a Schauder basis of $\mathcal{H}_n$.
\end{enumerate}
From this, the statement follows immediately.\\
We proceed by induction. For $n=1$, we know that all the conditions are satisfied due to Thm.~\ref{thm:AllHF} (the part about the order at $0$ is follows from the construction of the $H_1^k$). Now suppose we have found a suitable family $(\widehat{H_m^k})$ for $k\in\Z^+, m=1,\dots,n$, and pick any $k$. We want to construct a suitable $\widehat{H_{n+1}^k}$. As $H_1^k=\widehat{H_1^k}$, we know that \begin{align}
    \triangle^{n-1}H_n^k=\triangle^{n-1}\widehat{H}_n^k,
\end{align} 
that is, $G_{n-1}^k:=H_n^k-\widehat{H}_n^k\in\mathcal{H}_{n-1}$. By the induction hypothesis, we can therefore write $G_{n-1}^k$ as some countable sum of the $\widehat{H}_{n-1}^l$, for which we can then find a pre-image $G_n^k$ under $\triangle$ by substituting $n$ for $n-1$ in that sum representation. Note that, again by the induction hypothesis, we do not have any convergence issues here. We therefore have \begin{align}
    \triangle \left[H_{n+1}^k-G_n^k\right]=H_n^k-\left(H_n^k-\widehat{H}_n^k\right)=\widehat{H}_n^k.
\end{align}
Now all we need to do is to add a suitable harmonic function $J_{n,k}$ killing the boundary terms up to sufficiently high order (see the proof of Thm.~\ref{thm:AllHF}), and then, letting \begin{align}
    \widehat{H}_{n+1}^k:=H_{n+1}^k-G_n^k-J_{n,k}
\end{align}
we see that the first four conditions are satisfied (note that the bound on the degree at $0$ follows directly from the degree of the boundary conditions as well as the functional equation (\ref{eq:FE})). It therefore remains to show that the $\widehat{H}_m^k$ with $1\leq m\leq n+1, k\geq 1$ form a Schauder basis of $\mathcal{H}_{n+1}$. First, we will show that we can generate all polyharmonic functions. To see this, pick any element $H'_{n+1}\in\mathcal{H}_{n+1}$. By induction, we know that we can write \begin{align}
    \triangle H'_{n+1}=\sum_{m=1}^n\sum_{l\geq 1}a_{m,i}H_m^l.
\end{align}
However, letting $H_{n+1}:=\sum_{m=1}^n\sum_{l\geq 1}a_{m,i}H_{m+1}^l$, we have \begin{align}
    \triangle \left[H'_{n+1}-H_{n+1}\right]=0,
\end{align}
thus $H'_{n+1}-H_{n+1}\in \mathcal{H}_1$. Hence, up to addition of a harmonic function (remember that by Thm.~\ref{thm:AllHF} they can all be written as countable sum of the $H_1^k$), we know that our $H'_{n+1}$ can already be written as countable sum of the $\widehat{H}_m^l$, and we are done.\\
To show uniqueness of the coefficients, suppose we have two representations \begin{align}
    H_{n+1}=\sum_{m=1}^{n+1}\sum_{k=1}^\infty a_{m,k}H_m^k=\sum_{m=1}^{n+1}\sum_{k=1}^\infty b_{m,k}H_m^k.
\end{align}
As \begin{align}
    \triangle^nH_{n+1}=\sum_{k=1}^\infty (a_{n+1,k}-b_{n+1,k})H_1^k=0,
\end{align}
we know by the basis property for $\mathcal{H}_1$ that $a_{n+1,k}=b_{n+1,k}$ for all $k$. However, this means that \begin{align}
    \sum_{m=1}^n\sum_{k=1}^\infty a_{m,k}H_m^k=\sum_{m=1}^n\sum_{k=1}^\infty b_{m,k}H_m^k,
\end{align}
and by the induction hypothesis we therefore have $a_{m,k}=b_{m,k}$ for all $1\leq m\leq n+1,k\geq 1$.
\end{proof}
\textbf{Remarks:} 
\begin{itemize}
    \item If instead of defining the $H_1^k$ as in Thm.~\ref{thm:AllHF} one allows for any Schauder basis of $\mathcal{H}_1$, then via Lemma~\ref{lemma:criterion_AllPHF} one obtains not only a sufficient, but also a necessary condition for a family of polyharmonic functions to form a Schauder basis.
    \item In the constructions of polyharmonic functions given below, it will often be the case that it is clear from construction that for any given $n$, any sum of the form $\sum_{k\geq 1}a_kH_n^k$ converges as power series. In this case, we can let $\widehat{H}_n^k:=H_n^k$, and the proof boils down to the very last step.
    \item In Sections~\ref{sec:DiscretePHF} and \ref{sec:Decoupling}, different bases will be constructed. It is not immediate how to switch between them; the critical step is the construction of harmonic functions eliminating the boundary terms, which are generally not easy to write down explicitly. One could argue that the basis constructed in Section~\ref{sec:Decoupling} is in some sense a canonical one, seeing as it has a particularly nice shape, but it is not clear whether something comparable exists in the infinite group case, nor whether they stand out due to, say, a particular combinatorial interpretation.
    \end{itemize}\vskip\baselineskip

If we compare the functional equation (\ref{eq:FE}) for harmonic and polyharmonic functions, then the only difference lies in the additional term of $xyH_n(x,y)$ on the right-hand side not vanishing for the latter. In terms of the boundary value problem, this means that
we now want to solve \begin{small}\begin{align}\label{eq:DecouplingBVP}
    K\left(X_+,0\right)H_n\left(X_+,0\right)-K\left(X_-,0\right)H_n\left(X_-,0\right)=X_+yH_{n-1}(X_+,y)-X_-yH_{n-1}(X_-,y).
\end{align}
\end{small}\noindent
In an ideal world, the right-hand side of the latter equation would be $0$ as in the harmonic case, and this is indeed what happens for the simple walk (and, more generally, if ever $\pi/\theta=2$ and the group is finite, as will be discussed in Section~\ref{sec:pitheta2}). In this case, we can proceed as before, and obtain an explicit formula for polyharmonic functions, see Thm.~\ref{thm:pitheta2:finitegroup_PHF}.

\subsection{Example: the simple walk}\label{sec:SimpleWalk}
The simple walk has the step set $\mathcal{S}=\{\uparrow, \rightarrow,\downarrow,\leftarrow\}$, each with probability $\frac{1}{4}$. We have \begin{align}
    K(x,y)=\frac{xy}{4}\left(x+y+x^{-1}+y^{-1}\right)-xy,\quad
    \omega(x)=\frac{-2x}{(1-x)^2},\quad 
    \omega\left(X_+\right)=-\omega(y).
\end{align}
It turns out that the right-hand side of (\ref{eq:DecouplingBVP}) keeps vanishing, and thus one can iteratively construct polyharmonic functions via $H_{n+1}(x,y):=\frac{xyH_n(x,y)-X_+yH_n(X_+,y)}{K(x,y)}$. This allows us to find an explicit expression for all resulting polyharmonic functions. This property is directly tied to the fact that $\pi/\theta=2$, where $\theta$ is given by (\ref{eq:defTheta}), which will be discussed in more detail in Section~\ref{sec:pitheta2}.
\\
We therefore obtain a basis of all polyharmonic functions by letting \begin{align}
    H_{n+1}^k(x,y)=\frac{xyH_n^k(x,y)-X_+yH_n^k(X_+,y)}{K(x,y)}.
\end{align}
For an explicit formula as well as a proof, see  Thm.~\ref{thm:pitheta2:finitegroup_PHF}.
We can for instance compute 
$H_1^1=-\frac{8}{(1-x)^2(1-y)^2}$, $H_1^2=-\frac{32y}{(x-1)^2(y-1)^4}$, $H_1^3=-\frac{128y^2}{(x-1)^2(y-1)^6}$. One can show by induction that we have $H_1^k=-\frac{2\cdot 2^{2k}}{(x-1)^2(y-1)^{2k}}$.\\
Proceeding to compute the generating functions $V_{1,2,3}$ of $v_{1,2,3}$ as given in (\ref{eq:as_SW_1})--(\ref{eq:as_SW_2}), we obtain \begin{align*} 
&V_1=64H_1^1,\quad V_2=\frac{3}{8}H_2^1-\frac{3}{8}H_1^2+60H_1^1,\\ 
&V_3=-24H_3^1+24H_2^2+72H_2^1-30H_1^3-72H_1^2+5072H_1^1.
\end{align*}

It is somewhat striking here that the only $p$-polyharmonic part contained in $v_p$ is $H_p^1$, which is in some manner the simplest possible. At this stage there is neither a proof that this is always true nor a counter-example.

\section{A general solution}\label{sec:generalSolution}
While the computation for the simple walk turned out to be fairly simple, this was mainly due to the right-hand side of (\ref{eq:DecouplingBVP}) consistently vanishing. This does not happen in general. For the tandem walk, for instance, we arrive at 
\begin{align}
    K(X_+,0)H_1^1(X_+,0)-K(X_-,0)H_1^1(X_+,0)=\frac{y^3\sqrt{1-4y}}{(y-1)^5}.
\end{align}
The direct approach using a BVP like in the harmonic case does not generally yield an explicit solution as easily as before. One could modify the structure in order to obtain a similar BVP as before, using a decoupling function, which is the approach which works with finite group models and will be discussed in Section~\ref{sec:Decoupling}. However, we will first construct polyharmonic functions directly utilizing the functional equation (\ref{eq:FE}), independently of whether or not the group is finite. The main idea again utilizes that in Thm.~\ref{thm:AllHF}, we showed that given any power series $P(x)\in\mathbb{C}[[x]]$, we can construct a harmonic function $H(x,y)$ such that $K(x,0)H(x,0)=P(x)$. Now suppose that for one of the $H_1^k$ constructed in the aforementioned theorem, there is a $H_2^k$ such that $\triangle H_2^k=H_1^k$. Then, subtracting a harmonic function with the same values on the boundary, we know that there is also a $\widehat{H}_2^k$, such that $\triangle \widehat{H}_2^k=H_1^k$ and $K(x,0)\widehat{H}_2^k(x,0)=0$. By Lemma~\ref{lemma:criterion_AllPHF}, if we know this $\widehat{H}_2^k$ then we can reconstruct $H_2^k$ (or indeed any other biharmonic function with image $H_1^k$ under $\triangle$). Therefore, our strategy will be to utilize (\ref{eq:FE}) in order to find this particular $\widehat{H}_2^k$, where the assumption that $K(x,0)\widehat{H}_2^k(x,0)=0$ simplifies the equation immensely. While the construction itself is not very complicated, we will need a small technical lemma to make sure we will indeed end up with bivariate power series.
\begin{lemma}\label{lemma:Weierstrass} Suppose we have a model such that $K(0,0)=0$, $\left.\frac{\partial}{\partial x} K\right|_{x=y=0}\neq 0$, and select $X_+(y)$ such that $0=X_+(0)$. Furthermore, let $F(x,y)$ be a bivariate power series, such that $F(X_+(y),y)=0$ and that $F(x,0)\neq 0$. Then, \begin{align}
\frac{F(x,y)}{K(x,y)}\end{align}
is a bivariate power series in $x,y$.
\end{lemma}
\begin{proof}
By the Weierstraß preparation theorem, we can write \begin{align}
\label{eq:weierstrass1}
K(x,y)=e(x,y)(x-g(y)),\end{align}
with $e$ being an invertible bivariate power series, and $f_2(y)\in\C[[y]]$, with $f_2(0)=0$.\\
We can also rewrite \begin{align}\label{eq:weierstrass2}
F(x,y)=f(x,y)(x^k+x^{k-1}f_{k-1}(y)+\dots+f_0(y))=:f(x,y)P(x,y),\end{align}
where again $f$ is an invertible bivariate power series, the $f_i(y)\in\C[[y]]$ satisfy $f_i(0)=0$, and $P(x,y)\in\C[[y]][x]$ is a polynomial in $x$ over the ring $\C[[y]]$. Consequently, we have \begin{align}
\frac{F(x,y)}{K(x,y)}=\frac{f(x,y)}{e(x,y)}\cdot\frac{P(x,y)}{x-g(y)}.\end{align}
Since $e(x,y)$ is invertible, it remains to show that the second factor is a power series in $x,y$. To do so, all we need to do is to show that $g(y)$ is a zero of the polynomial $P(x,y)$, i.e. that $P(g(y),y)=0$.\\
By (\ref{eq:weierstrass2}), we know that $P(X_+(y),y)=0$ locally around $0$, and by (\ref{eq:weierstrass1}) we see that we also have $X_+(y)=g(y)$. The statement follows.
\end{proof}

\textbf{Remark:} The case $K(0,0)\neq 0$ (which is equivalent to our model having a North-East step) is not really of interest, as in this case $1/K(x,y)$ is a power series anyway, so the statement holds trivially. Also note that we must have either $K(0,0)\neq 0$, $\left.\frac{\partial}{\partial x} K\right|_{x=y=0}\neq 0$ or $\left.\frac{\partial}{\partial y} K\right|_{x=y=0}\neq 0$, because otherwise our model could have no North, North-East or East steps and would therefore be singular.

\begin{lemma}\label{lemma:ConstructPHF}
Suppose we have an arbitrary non-singular model with small steps and zero drift such that either $K(0,0)\neq 0$ or $\frac{\partial}{\partial x}K(x,y)\mid_{x=y=0}\neq 0$. Given any bivariate power series $G(x,y)$ which is analytic around $(0,0)$, we can then construct a power series $H(x,y)$ such that
\begin{enumerate}
    \item $H(x,y)$ is analytic around $(0,0)$,
    \item $\triangle H(x,y)=G(x,y)$, and
\item $K(x,0)H(x,0)=0$
\end{enumerate}
by letting \begin{align}
\label{eq:ConstructPHF_def1}
    H(x,y)&:=\frac{X_+y G(X_+,y)-xyG(x,y)}{K(x,y)}&\text{ if }K(0,0)=0,\\
    \label{eq:ConstructPHF_def2}
    H(x,y)&:=\frac{xyG(x,y)}{K(x,y)}&\text{ if }K(0,0)\neq 0,
\end{align}
where we select $X_+(y)$ such that $X_+(0)=0$.
\end{lemma}
\begin{proof}
The property $\triangle H=G$ can be written in terms of the functional equation (\ref{eq:FE}): \begin{align}
    K(x,y)H(x,y)=K(x,0)H(x,0)+K(0,y)H(0,y)-K(0,0)H(0,0)-xyG(x,y).
\end{align}
The case $K(0,0)\neq 0$ is easy, because then $1/K$ is a power series around $(0,0)$ and one can directly check that (\ref{eq:FE}) is satisfied. Consider now the case $K(0,0)=0$, and define $K(x,y)H(x,y)$ via (\ref{eq:ConstructPHF_def1}). The substitution $G(X_+,y)$ is valid because $X_+(0)=0$. One can check immediately that $K(x,0)H(x,0)=0$, and that (\ref{eq:FE}) is satisfied.\\
All that therefore remains to do is to show that we can divide the thusly obtained power series $K(x,y)H(x,y)$ by $K(x,y)$. To do so, we can utilize Lemma~\ref{lemma:Weierstrass}. To check the conditions to apply this lemma, note that $K(X_+,y)H(X_+,y)=0$ by construction. In order to satisfy the second condition, let $H'(x,y)$ be a harmonic function such that $K(x,0)H'(x,0)\neq 0$. We then have \begin{enumerate}
    \item $K(X_+,y)H(X_+,y)+K(X_+,y)H'(X_+,y)=0$, as is clear for the first summand from (\ref{eq:ConstructPHF_def1}) and for the second due to the fact that $K(X_+,y)=0$,
    \item $K(x,0)H(x,0)+K(x,0)H'(x,0)\neq 0$.
\end{enumerate}
We can therefore apply Lemma~\ref{lemma:Weierstrass} to the function $K(x,y)H(x,y)+K(x,y)H'(x,y)$. This tells us that the expression \begin{align}
    \frac{K(x,y)H(x,y)+K(x,y)H'(x,y)}{K(x,y)}=&\frac{K(x,y)H(x,y)}{K(x,y)}+\frac{K(x,y)H'(x,y)}{K(x,y)}\\
    =&H(x,y)+H'(x,y)
\end{align}
is a power series in $x,y$ around $(x,y)=(0,0)$. As we already know that $H'(x,y)$ is a power series, by consequence so is $H(x,y)$. Hence, we are done.
\end{proof}
As previously remarked, after potentially swapping $x$ and $y$ such that $\left.\frac{\partial}{\partial x} K\right|_{x=y=0}\neq 0$, this covers  all non-singular models with small steps and zero drift. Therefore, in the following we can assume without loss of generality that if $K(0,0)=0$, then $\frac{\partial}{\partial x}K(x,y)\mid_{x=y=0}\neq 0$. Utilizing Lemma~\ref{lemma:ConstructPHF}, it is now easy to construct a Schauder basis of all polyharmonic functions.
\begin{theorem}\label{thm:ConstructPHF}
Given a model with small steps and zero drift, and let \begin{align}\label{eq:defH1k}
    H_1^k(x,y):=\frac{P_k(\omega(x))-P_k\left(\omega(X_+)\right)}{K(x,y)},
\end{align}
where the polynomials $P_k$ are defined as in the proof of Thm.~\ref{thm:AllHF}. Then we can inductively construct bivariate power series $H_m^k$ via \begin{align}
\label{eq:definePHF1}
    H_{m+1}^k(x,y)&:=\frac{X_+ y H_m^k(X_+,y)-xyH_m^k(x,y)}{K(x,y)}\quad &\text{ if }K(0,0)=0,\\
    \label{eq:definePHF2}
    H_{m+1}^k(x,y)&:=\frac{xyH_m^k}{K(x,y)}\quad &\text{ if }K(0,0)\neq 0.
\end{align}
Each $H_m^k$ is $m$-polyharmonic, and we have $\triangle H_{m+1}^k=H_m^k$. 
\end{theorem}
\begin{proof}
By iterative application of Lemma~\ref{lemma:ConstructPHF}, one sees that the resulting expressions are power series in $x,y$. One can then easily check that the functional equation (\ref{eq:FE}) is satisfied. 
\end{proof}
\textbf{Remark:} instead of using the two different definitions (\ref{eq:definePHF1}), (\ref{eq:definePHF2}) depending on whether or not $K(0,0)=0$, one could just use (\ref{eq:definePHF1}) in any case. The disadvantage of that would be, however, fairly obvious: the resulting expressions are a bit more unwieldy, and we lose a bit of niceness (i.e. if $H_1^1(x,y)$ is rational, we would normally end up with an algebraic $H_2^1(x,y)$). Also, to make things work out formally one would still have to argue why all substitutions are valid.
\begin{theorem}
\label{thm:AllPHF_generalcase}
The polyharmonic functions $(H_m^k)_{m,k\in\N}$ constructed in Thm.~\ref{thm:ConstructPHF} form a basis of the space $\mathcal{H}$ of all polyharmonic functions.
\end{theorem}
\begin{proof}
By Lemma~\ref{lemma:criterion_AllPHF}.
\end{proof}

\subsection{Example: the tandem walk}\label{sec:example_TandemWalk_general}
The tandem walk is the model with step set $\mathcal{S}=\{\nwarrow, \rightarrow,\downarrow\}$, each with weight $\frac{1}{3}$. We find \begin{align}
    K(x,y)=\frac{x^2+y+xy^2}{3}-xy, \quad H_1^1(x,y)=\frac{81(xy-1)}{4(x-1)^3(y-1)^3},
\end{align}
leading to the harmonic function $h(i,j)=(i+1)(j+1)(i+j+2)$. We now want to find a biharmonic function $H_2^1$ such that $\triangle H_2^1=H_1^1$. To do so, we apply the procedure from Thm.~\ref{thm:AllPHF_generalcase}. First, we notice that $K(0,0)=0$, and that $\left.\frac{\partial K}{\partial x}\right|_{x=y=0}= 0$, while on the other hand we have $\left.\frac{\partial K}{\partial x}\right|_{x=y=0}=\frac{1}{3}\neq 0$. This is due to the fact that our model has no West, but a South step. Therefore, we need to swap the roles of $x$ and $y$ in (\ref{eq:ConstructPHF_def1}). We pick our $Y_+$ such that $Y_+(0)=0$; which gives us \begin{align}
    Y_+(x)=\frac{3x-1+(1-x)\sqrt{1-4x}}{2x},
\end{align}
and obtain \begin{footnotesize}\begin{multline}
    H_2^1(x,y)=\\
    -\frac{243((a-3)x^3-2xy+4x^2y+(7-3a)x^3y+2x^2y^2+(3a-13)x^3y^2+2x^4y^2+(3-a)x^3y^3)}{8(x-1)^5(y-1)^5(x^2+y+xy^2-3xy)},
\end{multline}
\end{footnotesize}
where $a:=\sqrt{1-4x}$. One can check that this expression is indeed a power series, and that $H_2^1(0,y)=0$. In particular, we have \begin{align}
    H_2^1(x,y)=\frac{243}{4}x+\frac{729}{4}xy+\frac{729}{2}xy^2+972 x^2y+\frac{5103}{4}x^3+\dots
\end{align}
As we already know, $H_2^1(x,y)$ is unique with the property $\triangle H_2^1=H_1^1$ only up to harmonic functions. And indeed, we will see in Section~\ref{sec:Decoupling} that instead of this algebraic function, there is a much nicer rational biharmonic function $\widehat{H_2^1}$ with $\triangle\widehat{H}_2^1=\widehat{H}_1^1$.

\subsection{Example: the king's walk}\label{sec:example_KingsWalk_general}
The king's walk is the model with step set $\mathcal{S}=\{\uparrow,\nearrow,\rightarrow,\searrow,\downarrow,\swarrow,\leftarrow,\nwarrow\}$, each with probability $\frac{1}{8}$. We find \begin{align}
    K(x,y)&=\frac{1+x+y+x^2+y^2+x^2y+xy^2+x^2y^2}{8}-xy,\\ H_1^1(x,y)&=-\frac{1}{16(x-1)^2(y-1)^2}.
\end{align}
As $K(0,0)\neq 0$, we can utilize (\ref{eq:ConstructPHF_def2}) and have \begin{align}
    H_2^1(x,y)=\frac{-128xy}{(x-1)^2(y-1)^2(1+x+y+x^2+y^2+x^2y+xy^2+x^2y^2-8xy)}.
\end{align}

\section{Continuous Polyharmonic Functions}\label{sec:continuousPHF}
Although we are primarily interested in discrete polyharmonic functions, it is still worthwhile looking at their continuous analogue. The original motivation for looking at discrete polyharmonic functions was via asymptotics of path-counting problems, and we know that the scaling limit of random walks is -- under some conditions, which however are all satisfied here, see \cite{Limic} -- a Brownian motion. We would therefore expect the scaling limit of a polyharmonic function in the quarter plane to be somehow related to that Brownian motion in $\R^+\times\R^+$ as well. Any such Browian motion is defined by its covariance matrix $\Sigma = \begin{pmatrix} \sigma_{11} & \sigma_{12}\\ \sigma_{12} & \sigma_{22}\end{pmatrix}$, and its infinitesimal generator is the Laplace-Beltrami operator \begin{align}\label{eq:defContLap}
    \triangle=\frac{1}{2}\left(\sigma_{11}\frac{\partial^2}{\partial x^2}+2\sigma_{12}\frac{\partial ^2}{\partial x\partial y}+\sigma_{22}\frac{\partial^2}{\partial y^2}\right).
\end{align} 
The coefficients $\sigma_{11},\sigma_{12},\sigma_{22}$ can be directly computed via $\mathbb{E}X^2=\sigma_{11},
    \mathbb{E}XY=\sigma_{12},
    \mathbb{E}Y^2=\sigma_{22}$ \cite{Limic}.
\begin{definition}\label{def:PHF_Continuous}
We call a  function $f$ a \textbf{(continuous) polyharmonic function of degree $k$} if \begin{align}
    \triangle^kf(x)&=0\quad\forall x\in\mathcal{W},\\
    f(x)&=0\quad \forall x\in\partial\mathcal{W},
\end{align}
where $\triangle$ is the Laplace-Beltrami operator given by (\ref{eq:defContLap}).
\end{definition}
Note that this definition is exactly the same as for discrete polyharmonic functions, except for the latter we used a discretization of the Laplacian. Also, while in the following there might be technically an ambiguity due to the same symbol $\triangle$ used for both the continuous and discrete Laplacian, it should always be clear from the context which one is to be used. \\
Similar to how discrete polyharmonic functions appear in the asymptotics of some path-counting problems, their continuous analogues occur when studying the asymptotics of exit times of Brownian motions \cite[VI.]{Aronszajn},\cite[Thm.~2.3]{Poly}.\\
The theory of continuous polyharmonic functions is well-developed, and computing them in a region as nice as the quarter plane is not a big challenge anymore. A common approach is to switch to polar coordinates, and then consider eigenfunctions of the resulting spherical Laplacian, as in \cite{Poly,Griffiths}. However, one can also make use of the fact that the two-dimensional Laplace transform $\mathcal{L}(f)$ of a polyharmonic function $f$ satisfies the following functional equation which can be seen as a continuous analogue of (\ref{eq:FE}), and was derived in \cite[App.~A]{Conformal}:
\begin{align}\label{eq:FE_continuous}
\gamma(x,y)\mathcal{L}(f)(x,y)=\frac{1}{2}\left[\sigma_{11}\mathcal{L}_1(f)(y)+\sigma_{22}\mathcal{L}_2(f)(x)\right]+\mathcal{L}(\triangle f)(x,y),
\end{align}
where we have\begin{small}\begin{align}
    \gamma(x,y)&=\frac{1}{2}\left(\sigma_{11}x^2+2\sigma_{12}xy+\sigma_{22}y^2\right),\quad L(f)(x,y)=\int_0^\infty\int_0^\infty e^{-ux-vy}f(u,v)\mathrm{d}u\mathrm{d}v,\\
    L_1(f)(y)&=\int_0^\infty \frac{\partial f}{\partial x}(0,v)e^{-vy}\mathrm{d}v,\quad\qquad\qquad
    L_2(f)(x)=\int_0^\infty\frac{\partial f}{\partial y}(u,0)e^{-ux}\mathrm{d}u,
\end{align}
\end{small}\noindent
see also \cite[2.2]{Poly}. This functional equation, which is very similar to (\ref{eq:FE}), is a rather large hint that there might be a connection between discrete and continuous polyharmonic functions; and in some cases one can use very similar methods to compute continuous polyharmonic functions. For the harmonic and biharmonic cases this has already been done via a direct computation in \cite[2.2]{Poly}; this will be generalized in Section~\ref{sec:Decoupling_continuous}.
\subsection{Relations between discrete and continuous cases}
The goal of this section is to give an overview of some of the similarities and general connections between discrete and continuous polyharmonic functions based on the functional equations (\ref{eq:FE}) and (\ref{eq:FE_continuous}).\vskip\baselineskip
For harmonic functions, since the last term of (\ref{eq:FE_continuous}) vanishes, everything works as in the discrete case, except the calculations turn out to be a lot simpler. We can define continuous versions $x_\pm$ of $X_\pm$, which satisfy $\gamma(x_\pm(y),y)=0$. It turns out that we have \begin{align}
    x_\pm(y)&=c_\pm y,\\
    c_\pm&=c e^{\pm i\theta},\\
    \hat{\omega}(x)&=\frac{1}{x^{\pi/\theta}},
\end{align}
where $\theta$ is the one mentioned in Section~\ref{sec:Prelims} and $c\in\R^+$. We can then construct (continuous) harmonic functions via \begin{align}\label{eq:defHF_continuous}
    \mathcal{L}(h_1^n)(x,y):=\frac{\omega(x)^n-\omega(x_+(y))^n}{\gamma(x,y)},
\end{align}
see also \cite[Thm.~2.4]{Poly}. Not very surprisingly, there is a relation between the discrete and continuous polyharmonic functions constructed in this manner. For the computations here as well as in later sections, the following lemma will be useful:
\begin{lemma}\label{lemma:convergence_toolkit1}
We have \begin{align}\label{eq:convergence_toolkit1_1}
    \lim_{\mu\to 0}\frac{K\left(e^{-\mu x},e^{-\mu y}\right)}{\mu^2}&=\gamma(x,y),\\
    \label{eq:convergence_toolkit1_2}
    \lim_{\mu\to 0}X_\pm \left(e^{-\mu y}\right)&=1+c_\pm y+\mathcal{O}(y^2).
\end{align}
\end{lemma}
\begin{proof}
Both of the results follow by a direct computation, which however in the second case is somewhat tedious. The main idea there is to write \begin{align}
    X_+(e^{-z})X_-(e^{-z})&=\frac{\tilde{c}(e^{-z})}{\tilde{a}(e^{-z})},\\
    X_+(e^{-z})X_-(e^{-z})&=-\frac{\tilde{b}(e^{-z})}{\tilde{a}(e^{-z})},
\end{align}
with $X_\pm(y)$ the solutions of $K(\cdot,y)=0$ as defined in Section~\ref{sec:Prelims},
and then use the fact that $X_+(1)=X_-(1)=1$ in order to obtain defining equations for the first coefficients in a series expansion of $X_\pm(e^{-z})$. 
\end{proof}
\textbf{Remark:} While for (\ref{eq:convergence_toolkit1_1}) it can be seen that this is a direct consequence of the drift being zero, it would be interesting to know if there is a more intuitive, or geometric way to obtain (\ref{eq:convergence_toolkit1_2}) as well.\vskip 1\baselineskip
Comparing the discrete and continuous constructions of harmonic functions (\ref{eq:DefHF}) and (\ref{eq:defHF_continuous}), it is not very surprising that there is a clear relation between them. 
\begin{theorem}\label{thm:convergence_HF}
We have \begin{align}
    \lim_{\mu\to 0}\mu^{k\pi/\theta-2}H_1^k\left(e^{-\mu x},e^{-\mu y}\right)=\alpha \mathcal{L}(h_1^k)(x,y)
\end{align}
for some non-zero constant $\alpha$.
\end{theorem}
\begin{proof}
Using Lemma~\ref{lemma:convergence_toolkit1}, all we need to show is that \begin{align}
    \lim_{\mu\to 0}\mu^{k\pi/\theta} \left[P_k(\omega(x))-P_k(\omega(X_+))\right]=\hat{\omega}(x)^k-\hat{\omega}(x_+)^k.
\end{align}
But this follows immediately from the fact that $\omega(x)=\frac{\alpha+o(1)}{(1-x)^{\pi/\theta}}$ in a neighbourhood of $x=1$ \cite[2.2]{FR11} and the fact that $[z^k]P_k(z)=1$ (see the construction of $P_k(z)$ in the proof of Thm.~\ref{thm:AllHF}), as well as for the second term once again Lemma~\ref{lemma:convergence_toolkit1}. 
\end{proof}

Knowing that, in the sense of a scaling limit as in Thm.~\ref{thm:convergence_HF}, we know that the discrete kernel $K$ corresponds to the continuous kernel $\gamma$, and that the Laplace transform can be understood as the continuous analogue of a generating function, it would be reasonable to expect that the boundary term $K(x,0)H(x,0)$ corresponds in the same fashion to $\frac{\sigma_{22}}{2}\mathcal{L}_2(h)(x)$. In the following, we will see that this is indeed the case.
\begin{lemma}\label{lemma:boundary_convergence_1}
Suppose $h(u,v)$ and its derivatives up to order $2$ are of exponential order, i.e. their absolute value is asymptotically bounded by $e^{c(x+y)}$ for some constant $c$, with $h(u,0)=h(0,v)=0$. Then we have \begin{align}
    \lap_2(h)(x)&=\lim_{y\to\infty}y^2\lap(h)(x,y),\\
    \lap_1(h)(y)&=\lim_{x\to\infty}x^2\lap(h)(x,y).
\end{align}
\end{lemma}
\begin{proof}
We will only show the first equality, the second follows by symmetry. From the computation in \cite[App.~A]{Conformal}, we see that 
\begin{align}
    \lap(h)(x,y)&=\frac{1}{y^2}\left[\lap\left(\frac{\partial^2 h}{\partial v^2}\right)(x,y)+\lap_2(h)(x)\right]\quad\Leftrightarrow\\
    y^2\lap(h)(x,y)&=\lap\left(\frac{\partial^2h}{\partial v^2}\right)(x,y)+\lap_2(h)(x).
\end{align}
Therefore, all that remains to show is that $\lim_{y\to\infty}\lap\left(\frac{\partial^2 h}{\partial y^2}\right)=0$. But this follows by monotone convergence using the growth assumption on $\frac{\partial^2 h}{\partial v^2}$. 
\end{proof}
\textbf{Remark:} The condition that $h(u,v)$ and its derivatives are of exponential order will hold true for those polyharmonic functions which have their origins in asymptotics of exit times of Brownian motions.
\begin{lemma}\label{lemma:boundary_convergence_2}
Let $\alpha\in\R$, and $H(x,y), \lap(h)(x,y)$ be polyharmonic such that \begin{align}
    \lim_{\mu\to 0}\mu^{\alpha+2} H\left(e^{-\mu x},e^{-\mu y}\right)=\lap(h)(x,y).
\end{align}
Assume furthermore that $H(x,y)$ is algebraic, and that the restrictions of $H(x,y)$ at $x=0$ and $y=0$ are well-defined. 
Then we have \begin{align}
    \lim_{\mu\to 0}\mu^\alpha K(e^{-\mu x},0)H(e^{-\mu x},0)&=\frac{\sigma_{22}}{2}\lap_2(h)(x),\\
    \lim_{\mu\to 0}\mu^\alpha K(0,e^{-\mu y})H(0,e^{-\mu y})&=\frac{\sigma_{11}}{2}\lap_1(h)(y).
\end{align}
\end{lemma}
\begin{proof}
Using Lemma~\ref{lemma:boundary_convergence_1}, we have
    \begin{align}
    \lim_{\mu\to 0}\mu^\alpha KH(e^{-\mu x},0)&=\lim_{\mu\to 0}\lim_{y\to \infty}\mu^\alpha KH(e^{-\mu x},e^{-\mu y}),\\
    \frac{\sigma_{22}}{2}\lap_2(h)(x)&=\lim_{y\to\infty}\gamma(x,y)\lap(h)(x,y)\\
    &=\lim_{y\to \infty}\lim_{\mu\to 0}\mu^\alpha KH(e^{-\mu x},e^{-\mu y}).
\end{align}
Thus, all we need to show is that we can exchange the order of the two limits. But this follows by the algebraicity of $H(x,y)$.
\end{proof}
\textbf{Remark:} There is a marked difference between the discrete and continuous cases in terms of the value of formal solutions. In the discrete case we work with formal power series, i.e. every formal solution of the functional equation (\ref{eq:FE}) leads to an actual solution since we can just extract coefficients. In the continuous case, however, this is not so simple: there are formal solutions of (\ref{eq:FE_continuous}) which turn out not to have an inverse Laplace transform. It is always possible to utilize the method given in Section~\ref{sec:generalSolution} to obtain continuous (formal) solutions, by simply defining $\mathcal{L}(h_n^k)$ as the scaling limit -- with an appropriate scaling factor -- of the discrete polyharmonic function $H_n^k$, and then by some computations using Lemmas~\ref{lemma:boundary_convergence_1} and \ref{lemma:boundary_convergence_2} one can check that (\ref{eq:FE_continuous}) is indeed satisfied. But the resulting solutions do generally not allow for an inverse Laplace transform: due to the shape of the kernel, which we repeatedly divide by, we cannot usually find a region of the form $\{\Re(x)\geq u,\Re(y)\geq v\}$ where $\lap(h_i^k)(x,y)$ is finite. By \cite{2dimLaplace}, this implies that $\lap(h_i^k)(x,y)$ is not the Laplace transform of any function, nor of any distribution. This will be different for the method presented in Section~\ref{sec:Decoupling}.

\subsection{Example: the scaling limit of the tandem walk}
For the scaling limit of the tandem walk (see Example~\ref{sec:example_TandemWalk_general}), we have \begin{align}
    \gamma(x,y)=\frac{1}{3}\left(x^2-xy+y^2\right),\quad
    c_\pm=\frac{1\pm i\sqrt{3}}{2},\quad
    \widehat{\omega}(x)=\frac{1}{x^3}.
\end{align}
We obtain $\lap(h_1^1)(x,y)=\frac{\widehat{\omega}(x)-\widehat{\omega}(c_+y)}{\gamma(x,y)}=\frac{3(x+y)}{x^3y^3}$, and one can check immediately that \begin{align}
    \lim_{\mu\to 0}\mu^5H_1^1\left(e^{-\mu x},e^{-\mu y}\right)=\lap(h_1^1)(x,y)
\end{align}
for $H_1^1$ computed in Section~\ref{sec:example_TandemWalk_general}. We can then proceed to the scaling limit of $H_2^1$, which gives us the formal solution of (\ref{eq:FE_continuous})
\begin{align}
    \lap(h_2^1)(x,y)=\frac{3 x^3 + 2 x y^2 + 2 y^3}{ 
x^3 y^5 (x^2 - x y + y^2)},
\end{align}
of which one can check directly that there is no inverse Laplace transform. We will see in Section~\ref{sec:exampleTandemC} that this is an advantage of the construction done in Section~\ref{sec:Decoupling} using decoupling functions, where the scaling limit of the resulting biharmonic function will properly be the Laplace transform of a continuous biharmonic function.
\section{Decoupling}\label{sec:Decoupling}
While the method given in Section~\ref{sec:generalSolution} gives us a Schauder basis of all polyharmonic functions, the resulting basis is not ideal in two senses:
\begin{enumerate}
    \item They do not have a continuous analogue, as discussed at the end of Section~\ref{sec:continuousPHF};
    \item They are often more complicated than necessary; for the king's walk we obtained a rational function which is singular on some not so easily described curve in Section~\ref{sec:example_KingsWalk_general}, and for the tandem walk the functions constructed in Section~\ref{sec:example_TandemWalk_general} were not even rational. We will see that both of these models have a basis which is a lot nicer to work with.
\end{enumerate}
Remember that the main issue why computing polyharmonic functions is not as easy as computing harmonic functions is that the right-hand side of (\ref{eq:DecouplingBVP}) does not usually vanish, and therefore the BVP approach does not immediately work. But in some cases, one can circumvent this problem by utilizing what is called a decoupling function in \cite[Def.~4.7]{Tutte}. 
\begin{definition} Let $M(x,y)$ be an rational function in $x,y$. If we can find $F(x), G(y)$ such that \begin{align}\label{eq:def:Decoupling}
    F(x)+G(y)\equiv M(x,y)\quad\operatorname{mod} K(x,y),
\end{align}
then we say that $F$ is a \textbf{decoupling function} of $M$.
\end{definition}
Here, we say that $A(x)\equiv B(x)\operatorname{mod} K(x,y)$ if there are polynomials $N(x,y),D(x,y)$ such that $D(x,y)$ is not divisible by $K(x,y)$ and $A(x)-B(y)=\frac{N(x,y)}{D(x,y)}K(x,y)$.
\vskip 1\baselineskip

These decoupling functions are closely related to the concept of invariants as in \cite[Def.~4.3]{Tutte}. An example of a decoupling function will for instance be given in Section~\ref{sec:exampleTandem}. Let in the following $H'$ be polyharmonic, and $H$ be such that $\triangle H=H'$. By substitution into (\ref{eq:DecouplingBVP}), we directly find that for $F(x)$ a decoupling function of $xyH'(x,y)$ we have \begin{align}\label{eq:def:Decoupling2}
    K(X_+,0)H(X_+,0)-F(X_+)-\left[K(X_-,0)H(X_-,0)-F(X_-)\right]=0.
\end{align}
In other words, if one knows how to compute a decoupling function of $xyH'(x,y)$, then one can again let $K(x,0)H(x,0)-F(x)=P(\omega)$ for some entire function $P$; by the same arguments as for the BVP outlined in Section~\ref{sec:DiscretePHF} one will then eventually arrive at a solution for $H(x,y)$. In \cite[App.~C]{Poly}, a decoupling function is guessed using an ansatz (as illustrated in Section~\ref{sec:guessing}) in order to compute a biharmonic function for the tandem walk. It turns out, however, that such a decoupling function can be explicitly computed for any model as long as the so-called group of the corresponding step set is finite. This group is generated by the mappings 
\begin{align}
    \Phi: (x,y)\mapsto \left(x^{-1}\frac{\tilde{c}(y)}{\tilde{a}(y)},y\right),\quad 
    \Psi: (x,y)\mapsto \left(x,y^{-1}\frac{c(x)}{a(x)}\right).
\end{align}
One can easily see that $\Phi,\Psi$ are involutions, and depending on the order of $\Theta:=\Phi\circ\Psi$, the group can be either finite or infinite. This group has been of interest in the study of random walks for some time now, see e.g. \cite{MBM,Book,Singer}. In particular, every group element $\gamma$ has a representation either of the form $\gamma=\Theta^k$ or $\gamma=\Phi\circ\Theta^k$. We can define $\operatorname{sgn}\gamma=1$ in the first, and $\operatorname{sgn}\gamma=-1$ in the second case. \vskip 1\baselineskip

\textbf{Remark:} In the literature, the group being finite is sometimes used ambiguously: here, as well as for example in \cite{MBM2,MBM}, the group being finite means it is finite as a group generated by the two birational transformations $\Phi,\Psi$. Some other times (e.g. in \cite{Book,Singer}), the group is understood as the restriction of these transformations to the curve $\mathcal{C}:=\{(x,y)\in \overline{\C}^2: K(x,y)=0\}$, where only the restriction of $\mathcal{G}$ on $\mathcal{C}$ would need to be finite, which is a weaker statement, and equivalent to the fact that $\pi/\theta\in\Q$ \cite[7.1]{Book}. These two notions are indeed different, as can be seen e.g. in Example~\ref{sec:pitheta2_infinite_example}. In this article, the group being finite means it being finite in the stronger sense, that is as a group of birational mappings on all of $\overline{\C}^2$.
\begin{theorem}[{see \cite[Thm.~4.11]{Tutte}}]\label{thm:Tutte}
Suppose our step set has a finite group of order $2n$, and $M(x,y)$ is rational such that \begin{align}\label{eq:SignedOrbit}
    \sum_{\gamma\in\mathcal{G}}\operatorname{sgn}(\gamma)\gamma\left(M(x,y)\right)=0.
\end{align}
Then a rational decoupling function of $M(x,y)$ is given by \begin{align}\label{eq:DecouplingFormula}
    F(x)=-\frac{1}{n}\sum_{i=1}^{n-1}\Theta^i\left[M(x,Y_+)+M(x,Y_-)\right].
\end{align}
\end{theorem}
In the following, we will show that $xyH_n(x,y)$ will turn out to have an orbit sum of $0$ for any polyharmonic $H_n$. This is in particular independent of whether or not the given model has a vanishing orbit sum as in \cite{Tutte}. 

\begin{corollary}\label{prop:oSum}
Suppose the group of the step set is finite and has a series representation around $(0,0)$. Then any algebraic function $M(x,y)$ of the form \begin{align}
    M(x,y)=xy\frac{u(x)+v(y)}{K(x,y)}
\end{align}
allows for a decoupling function via (\ref{eq:DecouplingFormula}).
\end{corollary}
\begin{proof}
For any point $(x,y)$ such that $K(x,y)\neq 0$, (\ref{eq:SignedOrbit}) is satisfied, seeing as the denominator $\frac{1}{xy}K(x,y)$ is invariant under $\mathcal{G}$, and alternating orbit summation over the numerator leads to a telescopic sum. As the set $\{(x,y): K(x,y)\neq 0\}$ is dense and $M(x,y)$ is algebraic, this implies that (\ref{eq:SignedOrbit}) is satisfied everywhere. By Thm.~\ref{thm:Tutte}, we can therefore construct a decoupling function via (\ref{eq:DecouplingFormula}).
\end{proof}
If a model has a finite group, then it can be shown that $\pi/\theta\in\mathbb{Q}$ (cf \cite[7.1]{Book}). The main difference between $\pi/\theta$ being integer or not is that in the former case, the conformal mapping $\omega(x)$ will be rational, and thus we can construct a basis consisting of rational functions.\\
To make things work out nicely in this case, we need to start wich a small technical lemma.
\begin{lemma}\label{lemma:decoupling_dividekernel}
    Let $N(x,y)$ be a polynomial such that $N(X_+,y)=N(X_-,y)=0$. Then, $K(x,y)|N(x,y)$.
    \end{lemma}
    \begin{proof}
    First we note that $N(X_+,y)=N(X_-,y)=0$ implies that also $N(x,Y_+)=N(x,Y_-)=0$ via the substitution $x\mapsto X_\pm(y)$. Using the notation of Section~\ref{sec:continuousPHF} of $\sigma_{11}=\mathbb{E}(X^2),\sigma_{22}=\mathbb{E}(Y^2)$, we can therefore assume that $\sigma_{11}\geq\sigma_{22}$, else we switch the roles of $x$ and $y$ in the following.\\
    We write $K(x,y)=(x-X_+)(x-X_-)\tilde{a}(y)$. As $N(X_+,y)=N(X_-,y)=0$ for any $y$, we know that $N(x,y)$ contains a factor $(x-X_+)(x-X_-)=x^2-(X_- +X_+)x+X_-X_+=x^2+\frac{\tilde{b}(y)}{\tilde{a}(y)}x+\frac{\tilde{c}(y)}{\tilde{a}(y)}\in\C(y)[x]$. By assumption, $N(x,y)$ is a polynomial; thus it must also contain a factor $\frac{\tilde{a}(y)}{\gcd(\tilde{a}(y),\tilde{b}(y),\tilde{c}(y))}$. Therefore, it suffices to show that $\tilde{a}(y),\tilde{b}(y),\tilde{c}(y)$ have no common zero. We have \begin{align}
        \tilde{a}(y)&=p_{1,1}y^2+p_{1,0}y+p_{1,-1},\\
        \tilde{b}(y)&=p_{0,1}y^2-y+p_{0,-1},\\
        \tilde{c}(y)&=p_{-1,1}y^2+p_{-1,0}y+p_{-1,-1}.
    \end{align}
    Now suppose there is an $u$ such that $\tilde{a}(u)=\tilde{b}(u)=\tilde{c}(u)=0$. Adding these three equations gives \begin{align}\label{eq:decoupling_dividebykernel1}
        0=u^2\left[p_{1,1}+p_{0,1}+p_{-1,1}\right]-u\left[1-p_{1,0}-p_{-1,0}\right]+\left[p_{1,-1}+p_{0,-1}+p_{-1,-1}\right].
    \end{align}
    As the drift is $0$, we know that the coefficient of $u^2$ is the same as the constant, namely $\frac{1}{2}\sigma_{11}$, and the coefficient of $u$ is $\sigma_{22}$. Since our model is non-singular, we have $\sigma_{11},\sigma_{22}>0$. We can therefore rewrite (\ref{eq:decoupling_dividebykernel1}) as 
    \begin{align}
        \sigma_{11}u^2-2\sigma_{22}u+\sigma_{11}=0.
    \end{align}
    Using the quadratic formula, we obtain \begin{align}
        u=\frac{\sigma_{22}}{\sigma_{11}}\pm\sqrt{\left(\frac{\sigma_{22}}{\sigma_{11}}\right)^2-1}.
    \end{align}
    If $\sigma_{11}=\sigma_{22}$, then $u=1$, but we see that $\tilde{b}(1)<0\leq \tilde{a}(1),\tilde{b}(1)$.\\
    Therefore we must have $\sigma_{11}>\sigma_{22}$, so we have two complex conjugate solutions for $u$. However, as $\tilde{b}(0)>0$ and $\tilde{b}(1)<0$, we know that $\tilde{b}(y)$ can only have real solutions (note in particular that this does not change if $p_{0,1}=0$, in which case $\tilde{b}(y)$ is linear), so $\tilde{b}(u)=0$ cannot hold.\\
    Hence, $\tilde{a}(y),\tilde{b}(y),\tilde{c}(y)$ cannot have a common factor, and the statement follows.
    \end{proof}
Using the above lemma, we can now use decoupling functions to construct, in the case of a finite group with $\pi/\theta\in\Z$, rational discrete polyharmonic functions of a particularly nice shape.
\begin{theorem}\label{thm:buildPHF}
Suppose our step set has finite group and $\pi/\theta\in\mathbb{Z}$. Let $H_1^k(x,y)$ be defined by (\ref{eq:defH1k}). We can then define inductively \begin{align}\label{eq:buildPHF}
    H_{n}^k(x,y)=\frac{xyH_{n-1}^k(x,y)-F_{n-1}^k(x)-\left[X_+yH_{n-1}^k(X_+,y)-F_{n-1}^k(X_+,y)\right]}{K(x,y)},
\end{align}
where $F_n^k(x)$ is the decoupling function of $xyH_n^k(x,y)$ defined by (\ref{eq:DecouplingFormula}), which in particular exists. Then, $H_n^k(x,y)$ is a rational function in $\mathcal{H}_n$ for all $n,k$, which satisfies $\triangle H_{n+1}^k=H_n^k$. For each $n,k$ we can write \begin{align}\label{eq:formPHF}
    H_n^k(x,y)=\frac{p_{n,k}(x,y)}{(1-x)^{\alpha}(1-y)^{\alpha}},
\end{align}
where $p_{n,k}(x,y)$ is a polynomial and $\alpha\in\N$. 
\end{theorem}

\textbf{Remarks:} \begin{itemize}
    \item In Thm.~\ref{thm:Convergence} we will see that $\alpha\leq k\pi/\theta+2(n-1)$.
    \item Defining decoupling functions and utilizing them in order to compute polyharmonic functions works, as long as the group is finite, for any $\pi/\theta$ (which must then automatically be rational). In particular, one can check that an analogous version of (\ref{eq:DecouplingFormula}) holds. However, in the non-integer case we do not obtain polynomial functions anymore, and in particular we will not have a representation like (\ref{eq:formPHF}); the main reason being that an equivalent of Lemma~\ref{lemma:decoupling_dividekernel} does not hold. Therefore we lose information about the positioning of singularities, which will generally not only be where $x=1$ or $y=1$. The rest one can prove in the same manner as the corresponding points in the proof of Thm.~\ref{thm:buildPHF}. 
\end{itemize}

\begin{proof} We proceed by induction. In each step, we will show that:\begin{itemize}
    \item $H_n^k(x,y)$ is rational,
    \item $H_n^k(x,y)$ has its only poles at $x=1$ or $y=1$,
    \item $xyH_n^k(x,y)$ does not have a pole at $x=\infty$ or $y=\infty,$
    \item $xyH_n^k(x,y)$ has orbit sum $0$ and thus admits a decoupling function $F_n^k(x)$,
    \item $F_n^k(x)$ has its only pole at $x=1$.
\end{itemize} To see that $\triangle H_{n+1}^k(x,y)=H_n^k(x,y)$, one can simply plug (\ref{eq:buildPHF}) into the functional equation (\ref{eq:FE}).\\
So consider first the case $n=1$. $H_1^k(x,y)$ being rational follows immediately from $\pi/\theta\in\mathbb{Z}$, and thus $\omega$ being rational (see \cite[(3.12)]{Conformal}). As by construction the numerator $N_1^k(x,y)$ of $xyH_1^k(x,y)$ as defined in (\ref{eq:buildPHF}) satisfies $N_1^k(X_\pm)=0$, it must according to Lemma~\ref{lemma:decoupling_dividekernel} be a multiple of $K(x,y)$, thus the only poles of $H_1^k(x,y)$ can be those coming from $\omega(x),\omega(X_+)$. Since $\omega(x)$ has its only pole at $x=1$ and $X_+(y)=1$ only if $y=1$, $H_1^k(x,y)$ can only have poles at $x=1,y=1$. Similarly, we check by a direct computation that $xyH_1^k(x,y)$ does not have a pole at $x=\infty,y=\infty$. The existence of a decoupling function $F_1^k(x)$ follows immediately from Prop.~\ref{prop:oSum}. Finally, we can deduce from (\ref{eq:DecouplingFormula}), utilizing that $xyH_1^k(x,y)$ does not have poles at infinity and noting that $(1,1)$ is a fixed point under the group, that $F_n^k(x)$ has its only pole at $x=1$, thus the case $n=1$ is done.\\
Now let $n\geq 2$ and assume the theorem is already shown up to $n-1$. We then formally define as in (\ref{eq:buildPHF}) \begin{align}
    H_n^k(x,y):=\frac{xyH_{n-1}^k(x,y)-F_{n-1}^k(x)-\left[X_+yH_{n-1}^k(X_+,y)-F_{n-1}^k(X_+,y)\right]}{K(x,y)}.
\end{align}
First, we need to argue that $H_n^k(x,y)$ is rational. By assumption, we know that $H_{n-1}^k(x,y)$ is rational. To see that the remaining part of the numerator, that is, $X_+yH_{n-1}^k(X_+,y)-F_{n-1}^k(X_+)$, is rational, we use the defining property (\ref{eq:def:Decoupling}) of the decoupling function $F_n^k$, rewriting \begin{align}
    G(y)&=X_+yH_{n-1}^k(X_+,y)-F_{n-1}^k(X_+)&\Rightarrow\\
    xyH_{n-1}^k(x,y)&\equiv F_{n-1}^k(x)+X_+yH_{n-1}^k(X_+,y)-F_{n-1}^k(X_+)&\mod K(x,y),
\end{align}
and since $xyH_{n-1}^k(x,y)$, $F_{n-1}^k(x)$ as well as $K(x,y)$ are rational, so is $X_+yH_{n-1}^k(X_+,y)-F_{n-1}^k(X_+)$. Next, we consider the poles of $H_n^k(x,y)$. Again, by construction we have that the numerator $N_n^k(X_\pm,y)=0$, and thus by Lemma~\ref{lemma:decoupling_dividekernel} the $K(x,y)$ in the denominator cancels. As we know by assumption that $F_{n-1}^k(x)$ has its only pole at $x=1$, and as $X_+(y)=1$ if and only if $y=1$, there will be no new poles coming from the $F_{n-1}^k$-parts. The same goes for $H_{n-1}^k(x,y)$ and $H_{n-1}^k(X_+,y)$, and therefore the only poles of $H_n^k(x,y)$ can be at $x=1,y=1$. \\
To check that $xyH_n^k(x,y)$ does not have a pole at infinity, we utilize (\ref{eq:FE}):
\begin{align*}
    K(x,y)H_n^k(x,y)=K(x,0)H_n^k(x,y)+K(0,y)H_n^k(0,y)-xyH_{n-1}^k(x,y).
\end{align*}
If $xyH_n^k(x,y)$ had a pole at infinity, then so would $K(x,y)H_n^k(x,y)$. But as by the induction hypothesis, $xyH_{n-1}^k(x,y)$ does not have a pole at infinity; so the pole for, say, $x\to\infty$ of the left-hand side would need to cancel with $K(x,0)H_n^k(x,0)$ on the right-hand side. But the left-hand side depends on $y$ while the $K(x,0)H_n^k(x,0)$ doesn't, so the poles cannot cancel for all values of $y$; a contradiction.\\
To see that a decoupling function of $xyH_n^k(x,y)$ exists, we split $xyH_n^k(x,y)$ in two parts. First, we notice that \begin{align}
    xy\frac{xyH_{n-1}^k(x,y)}{K(x,y)}=\frac{xy}{K(x,y)}xyH_{n-1}^k(x,y),
\end{align}
and since $\frac{xy}{K(x,y)}$ is invariant under the group and we already know that $xyH_{n-1}^k(x,y)$ has a decoupling function (and thus its orbit sum is $0$), we deduce that this part as well has orbit sum $0$, and thus it can be decoupled by Thm.~\ref{thm:Tutte}. For the rest, we notice that \begin{align}
    xy\frac{-F_{n-1}^k(x)-\left[X_+yH_{n-1}^k(X_+,y)-F_{n-1}^k(X_+,y)\right]}{K(x,y)}
\end{align}
has the form $xy\frac{A(x)+B(y)}{K(x,y)}$ (note that $X_+(y)$ does in fact not depend on $x$), and thus its orbit sum is $0$ by Cor.~\ref{prop:oSum}. Therefore, Thm.~\ref{thm:Tutte} gives us a decoupling function $F_n^k(x)$ of $xyH_n^k(x,y)$ via (\ref{eq:DecouplingFormula}). As each summand has its poles at $x=1$ only, and $\Theta$ leaves the point $(1,1)$ invariant, we know that $F_n^k$ has its only pole at $x=1$.\\
It remains to show that the order of the poles at $x,y=1$ is at most $k\cdot \pi/\theta + 2(n-1)$. For $n=1$ this can again be verified directly; afterwards it follows by induction: by a short computation one can see that the order of the pole of $F(x)$ compared to the one at $x=1$ of $xyH(x,y)$ increases at most by $2$, and by a similar argument for the $G(y)$ in (\ref{eq:def:Decoupling}) (see \cite[Thm.~4.11]{Tutte} for an explicit formula) one can show the same for $X_+yH(X_+,y)-F(X_+)=G(y)$. Using (\ref{eq:buildPHF}) finally yields the statement.
\end{proof}
By Lemma~\ref{lemma:criterion_AllPHF}, it therefore follows that the thusly constructed polyharmonic functions form a Schauder basis of the space of all polyharmonic functions.

\subsection{Example: the tandem walk revisited}\label{sec:exampleTandem}
To illustrate the results from Section~\ref{sec:Decoupling}, consider once again the tandem walk, which has the step set $\mathcal{S}=\{\rightarrow, \downarrow, \nwarrow\}$, with weights $\frac{1}{3}$ each. As in Example~\ref{sec:example_TandemWalk_general}, we have
\begin{align}
    K(x,y)=\frac{xy}{3}\left(x^{-1}+y+xy^{-1}\right)-xy,\quad
    H_1^1(x,y)=\frac{81(xy-1)}{4(x-1)^3(y-1)^3}.
\end{align}
Coefficient extraction then led us to recover the original harmonic function from the generating series, giving us $h_1^1(i,j)=(i+1)(j+1)(i+j+2)$. In Section~\ref{sec:Decoupling} we computed a biharmonic function for $H_1^1$, which was however not rational. Using the method presented in this section, however, we will find that there is, in fact, a rational one.\\
First, one can check that the group is finite and of order $6$; we have 
\begin{align}
    (x,y)\stackrel{\Psi}{\mapsto}\left(x,\frac{x}{y}\right)\stackrel{\Phi}{\mapsto}\left(\frac{1}{y},\frac{x}{y}\right)\stackrel{\Psi}{\mapsto}\left(\frac{1}{y},\frac{1}{x}\right)\stackrel{\Phi}{\mapsto}\left(\frac{y}{x},\frac{1}{x}\right)\stackrel{\Psi}{\mapsto}\left(\frac{y}{x},y\right)\stackrel{\Phi}{\mapsto}(x,y).
\end{align}
Now using (\ref{eq:DecouplingFormula}), we obtain the decoupling function $F_1(x)=-\frac{81x^3}{4(1-x)^5}$. Note that this decoupling function is not the same one as is given in \cite[App.~C]{Poly}, where instead (after scaling) $F_1'=\frac{-81x^3}{4(1-x)^6}$ is given. This goes to show that the choice of a decoupling function is, due to the invariance property in (\ref{eq:wInv}), unique only up to functions of $\omega$; in this particular case we have (up to a multiplicative constant) $F_1'(x)-F_1(x)=\omega(x)^2$. The way in which this alternative decoupling function was found is described in Section~\ref{sec:guessing}.\\
We can now utilize this $F_1$ in order to compute a biharmonic function; (\ref{eq:buildPHF}) directly gives us \begin{align}H_2^1=-\frac{243(xy-1)(x+y+xy(x+y-4))}{(x-1)^5(y-1)^5},\end{align}
which after extracting coefficients corresponds to \begin{multline}
h_2^1(i,j)=(i+1)(j+1)(-36 i - 30 i^2 - 6 i^3 - 36 j - 44 ij - 14 i^2 j - 2 i^3 j - 
  30 j^2\\ - 14 ij^2 + i^2 j^2 + i^3 j^2 - 6 j^3 - 2 i j^3 + i^2 j^3 + 
  i^3 j^3)
\end{multline}We can now use (\ref{eq:DecouplingFormula}) again to obtain the next decoupling function $F_2(x)=\frac{81x^2(x+2)}{4(x-1)^7}$, which we can then use to compute \begin{align}H_3^1=\frac{p(x,y)}{(x-1)^7(y-1)^7},\end{align} where $p(x,y)$ is a somewhat unwieldy polynomial of degree $9$. 

\subsection{Example: the king's walk revisited}
Consider now once again the king's walk with the step set $\mathcal{S}=\{\uparrow,\nearrow,\rightarrow,\searrow,\downarrow,\swarrow,\leftarrow,\nwarrow\}$, each with probability $\frac{1}{8}$. We have, as in Section~\ref{sec:example_KingsWalk_general}, \begin{align}
    K(x,y)&=\frac{1+x+y+x^2+y^2+x^2y+xy^2+x^2y^2}{8}-xy,\\ H_1^1(x,y)&=-\frac{1}{16(x-1)^2(y-1)^2}.
\end{align}
After coefficient extraction, we find that $h_1^1(i,j)=(i+1)(j+1)$. While the biharmonic function we obtained in Section~\ref{sec:example_KingsWalk_general} was rational, it did not have a shape which made it very easy to describe its singularities, or to extract coefficients. This will once again be very different applying the decoupling method.\\
The king's walk has a finite group of order $4$, namely \begin{align}
    (x,y)\stackrel{\Psi}{\mapsto}\left(x,\frac{1}{x}\right)\stackrel{\Phi}{\mapsto}\left(\frac{1}{x},\frac{1}{y}\right)\stackrel{\Psi}{\mapsto}\left(\frac{1}{x},y\right)\stackrel{\Phi}{\mapsto}(x,y).
\end{align}
It turns out that in this case, we can pick $0$ as a decoupling function, as the right-hand side of (\ref{eq:DecouplingBVP}) vanishes. Therefore, (\ref{eq:buildPHF}) gives us $\hat{H}_2^1(x,y)=-\frac{128y}{3(x-1)^2(y-1)^4}$, which is essentially the same result as we obtained for the simple walk in Section~\ref{sec:SimpleWalk}. This is not a coincidence; in both cases we have $\pi/\theta=2$ and a finite group; hence we can apply Thm.~\ref{thm:pitheta2:finitegroup_PHF}, and from there on one easily sees that the resulting polyharmonic functions will be the same. 

\subsection{Continuous decoupling}\label{sec:Decoupling_continuous}
The idea of decoupling in the continuous setting, as suggested in~\cite{Conformal, Poly}, is very much the same as in Section~\ref{sec:Decoupling}. The continuous version of the BVP for polyharmonic functions now reads \begin{align}\label{eq:BVP_continuous}
    \sigma_{22}\lap_2(h_n)(c_+y)-\sigma_{22}\lap_2(h_n)(c_-y)=\lap(h_{n-1}(c_+y,y))-\lap(h_{n-1}(c_-y,y)).
\end{align}
Our goal will now be to construct a decoupling function $f_{n-1}(x)$, such that \begin{align}\label{eq:cont_decoupling}
f_{n-1}(c_+y)-f_{n-1}(c_-y)=\lap(h_{n-1})(c_+y,y)-\lap(h_{n-1})(c_-y,y).
\end{align}
One key point to note here is that all expressions appearing in (\ref{eq:BVP_continuous}) are homogeneous\footnote{That is, there is some $m\in\R$ such that they satisfy $T(\lambda x,\lambda y)=\lambda^mT(x,y)$ for all $x,y$; we call this $m$ the degree of $T$ and (written as $\deg T$).}, which follows for $\lap(h_1^k)$ by construction, and can be checked for the others by induction. In particular, this means that the right-hand side of (\ref{eq:cont_decoupling}) is homogeneous as well; and as it depends only on $y$ it must therefore be of the form $\alpha y^m$ for some $\alpha$, and $m=\deg\lap(h_{n+1})$. Consequently, we choose the ansatz $f_n=\beta y^m$, with $m=\deg\lap(h_n)$. The equation we wish to solve thus reads, provided $(c_+)^m\neq (c_-)^m$, \begin{align}
    \beta\left[(c_+y)^m-(c_-y)^m\right]=\alpha y^m\quad\Leftrightarrow\quad \frac{\alpha}{(c_+)^m-(c_-)^m}=\beta.
\end{align}
Remembering that $c_\pm=ce^{\pm i\pi/\theta}$, we see that this is solveable in general only if $m$ is not an integer multiple of $\pi/\theta$, as then we would have $(c_+)^m-(c_-)^m=0$. As it turns out, this constraint does not in fact matter: whenever we would run into this issue, it just so happens that $\alpha$ is already $0$, i.e. we do not need a decoupling function (see Example~\ref{sec:exampleTandemC}). At this stage, no direct proof of this is known, and it would be very interesting to find a way to see this directly. But one can use the convergence properties of discrete polyharmonic functions to show that the decoupling function will be $0$ in all suitable cases to circumvent this problem. Since this is essential in order to continue the procedure but we will use convergence properties which will be introduced later, this  will be stated here and be proven in Section~\ref{sec:decoupling_convergence}. An illustration of this is given in Example~\ref{sec:exampleTandemC}.

\begin{lemma}\label{lemma:contDecoupling}
In the setting of Thm.~\ref{thm:PHF_C} below, if $(c_+)^m-(c_-)^m=0$, then the right-hand side of (\ref{eq:cont_decoupling}) vanishes. In particular, we can always find a decoupling function of the form $f_n^k(x)=\alpha x^m$, with $m=\deg\lap{h_n^k}(c_+x,x)$ (where $\alpha=0$ if $(c_+)^m=(c_-)^m$). 
\end{lemma}

Utilizing the above lemma, it is now easy to prove the continuous analogue of Thm.~\ref{thm:buildPHF}.
\begin{theorem}\label{thm:PHF_C}
Suppose we have a non-singular model with zero drift, small steps and such that $\pi/\theta\in\Z$. Let $\lap(h_1^k)(x,y)$ be defined by (\ref{eq:defHF_continuous}). We can then define inductively \begin{align}\label{eq:buildPHF_C}
    \lap(h_n^k)(x,y)=\frac{\lap(h_{n-1})(x,y)-f_{n-1}(x)-\left[\lap(h_{n-1})(c_+y,y)-f_{n-1}(c_+y)\right]}{\gamma(x,y)},
\end{align}
where $f_n(x)$ is a decoupling function as in (\ref{eq:cont_decoupling}). Then, $\lap(h_n^k)(x,y)$ is the Laplace transform of an $n$-harmonic function, such that $\mathcal{L}h_n^k=h_{n-1}^k$. For each $n,k$ we can write \begin{align}
    \lap(h_n^k)(x,y)=\frac{q_{n,k}(x,y)}{x^\alpha y^\alpha},
\end{align}
for $\alpha\in\N$ and  $q_{n,k}(x,y)$ a homogeneous polynomial. 
\end{theorem}
\textbf{Remark:} We will see in Thm.~\ref{thm:Convergence} that $\alpha\leq k\pi/\theta+2(n-1)$. 
\begin{proof}
For $n=1$, the statement can be checked directly. Now suppose the statement holds for $n$, thus we know that $\mathcal{L}(h_n^k)(x,y)=\frac{q_{n,k}(x,y)}{x^\alpha y^\alpha}$. By Lemma~\ref{lemma:contDecoupling} (which will be proven in Section~\ref{sec:decoupling_convergence}), we know that we can find a decoupling function, which must either be $0$ or have the same degree as $\mathcal{L}(h_n^k)(c_+y,y)$,  and we can therefore formally define $\lap(h_{n+1}^k)(x,y)$ via (\ref{eq:buildPHF_C}). One can check that each summand is homogeneous of the same degree; hence so is their sum. By construction, the numerator of (\ref{eq:buildPHF_C}) is $0$ for $x= c_\pm y$; it must therefore contain a factor $\gamma(x,y)=\sigma_{11}(x-c_+y)(x-c_-y)$, so the denominator cancels. The fact that $\lap(h_{n+1}^k)(x,y)$ is the Laplace transform of a continuous polyharmonic function such that $\triangle h_{n+1}^k(s,t)=h_n^k(s,t)$ follows from checking that the functional equation (\ref{eq:FE_continuous}) is satisfied, and noticing that we can perform an inverse transform on monomials of the form $x^uy^v$ for $u,v\in\R$. 
\end{proof}

\textbf{Remark:} while the construction of discrete polyharmonic functions via decoupling functions is only possible if the group is finite, there are no such restrictions in the continuous setting. 

\subsection{Example: the scaling limit of the tandem walk revisited}\label{sec:exampleTandemC}
For the scaling limit of the tandem walk, we have \begin{align}
    \gamma(x,y)=\frac{1}{3}\left(x^2-xy+y^2\right),\quad
    c_\pm=\frac{1\pm i\sqrt{3}}{2},\quad
    \widehat{\omega}(x)=\frac{1}{x^3}.
\end{align}
As before, we have $\lap(h_1^1)(x,y)=\frac{3(x+y)}{x^3y^3}$, and thus (\ref{eq:cont_decoupling}) takes the form 
\begin{align}
\label{eq:decoupling_continuous_tandem_1}f_1^1(c_+y)-f_1^1(c_-y)=\lap(h_1^1)(c_+y,y)-\lap(h_1^1)(c_+y,y)=\frac{3i\sqrt{3}}{y^5}.\end{align}
By a quick computation, one obtains $f_1(x)=\frac{-3}{x^5}$ and a biharmonic function
\begin{align}
    \lap(h_2^1)(x,y)&=\frac{\lap(h_1^1)(x,y)-f_1(x)-\left[\lap(h_1^1)(c_+y,y)-f_1^1(c_+y)\right]}{\gamma(x,y)}\\
    &=\frac{9(x+y)(x^2+y^2)}{x^5y^5}.
\end{align}
Performing the inverse Laplace transform, this gives us \begin{align}
    h_2^1(s,t)=-\frac{81}{8}st(s+t)(s^2+st+t^2).
\end{align}
For computing a triharmonic function, our decoupling function must now satisfy \begin{align}\label{eq:decoupling_continuous_tandem_2}
    f_2^1(c_+y)-f_2^1(c_-y)=\lap(h_2^1)(c_+y,y)-\lap(h_2^1(c_-y,y)=\frac{243i\sqrt{3}}{4y^7},
\end{align} 
which leads to $f_2^1(x)=-\frac{243}{x^7}$ and \begin{align}
    \lap(h_3^1)(x,y)&=\frac{729(x+y)(x^2-xy+y^2)(x^2+xy+y^2)}{4x^7y^7},
\end{align}
and \begin{align}
    h_3^1(s,t)=-\frac{81}{320}st(s+t)(s^2+st+t^2)^2.
\end{align}
When trying to compute a decoupling function $f_3^1(x)$ as in (\ref{eq:decoupling_continuous_tandem_1}) and (\ref{eq:decoupling_continuous_tandem_2}), seeing that the degree of the denominator will always increase by $2$, this is where one might expect things to go wrong, as $(c_+)^9=(c_-)^9$ and thus an ansatz as above might not work. However, doing the computation one finds that \begin{align}\label{eq:decoupling_continuous_tandem_3}
    f_3(c_+y)-f_3(c_-y)=\lap(h_3^1)(c_+y,y)-\lap(h_3^1)(c_-y,y)=0,
\end{align}
thus we can pick $f_3(x)=0$ and directly obtain a $4$-harmonic function \begin{align}
    \lap(h_4^1)(x,y)=\frac{2187 (x^3 + 2 x^2 y + 2 x y^2 + y^3)}{4 x^7 y^7},
\end{align}
which leads to \begin{align}
    h_4^1(s,t)=\frac{81}{640} s^3 t^3 (s + t)^3.
\end{align}
The fact that the right-hand side of (\ref{eq:decoupling_continuous_tandem_3}) turns out to be $0$ is a consequence of the convergence of discrete to continuous polyharmonic and decoupling functions, and will be shown in the next section in Thm.~\ref{thm:Convergence}.

\subsection{The scaling limit}\label{sec:decoupling_convergence}

Using Lemma~\ref{lemma:convergence_toolkit1}, the strategy to show a general convergence of the polyharmonic functions obtained by decoupling is quite simple: we use the fact that the recursive definitions (\ref{eq:buildPHF}) and (\ref{eq:buildPHF_C}) have the same structure, and take the limit of each term separately. All that remains to consider are the decoupling functions. However, using once again Lemma~\ref{lemma:convergence_toolkit1}, this turns out to be rather straightforward, too.

\begin{lemma}\label{lemma:limit_Decoupling}
Suppose we have are given discrete and continuous polyharmonic function $H(x,y)$ and $\lap(h)(x,y)$ respectively, and a constant $\alpha$ such that \begin{align}
    \lim_{\mu\to 0}\mu^\alpha H\left(e^{-\mu x},e^{-\mu y}\right)=\lap(h)(x,y).
\end{align}
Then, if we can construct a decoupling function $F(x)$ of $xyH(x,y)$ via (\ref{eq:DecouplingFormula}), the limit\begin{align}
    f(x):=\lim_{\mu\to 0} \mu^\alpha F\left(e^{-\mu x}\right)
\end{align}
exists and is a decoupling function of $\lap(h)(x,y)$. 
\end{lemma}

\textbf{Remark:} In the context of the construction in Thm.~\ref{thm:buildPHF}, we know that $F_n^k(x)$ is rational with its only pole at $x=1$. From this, we can conclude immediately that $f(x)$ will have the form $f(x)=\frac{\beta}{x^\alpha}$, where $\beta$ may or may not be $0$. This, as we will see, is essentially the idea of the proof of Lemma~\ref{lemma:contDecoupling}.

\begin{proof}
To see that the limit exists, we note that $\alpha$ must be the order of the pole at $x=y=1$ of $H(x,y)$, and thus also the order of the pole of $xY_\pm H(x,Y_\pm)$. Noticing that, due to (\ref{eq:DecouplingFormula}), $F(x)$ consists of such summands with powers of $\Theta$ applied to them, provided that $\Theta'(x,Y_\pm)\neq 0$, we know that the maximum possible order of the pole of $F(x)$ at $x=1$ is $\alpha$. The condition about the derivative, however, is guaranteed by the parametrization of the kernel curve we will use in Section~\ref{sec:parametrization}, which tells us that we have $\Theta(x(s))=s/q$, and therefore the derivative wrt $x$ can never be $0$. Thus, the limit exists, and the statement follows by taking the limit of (\ref{eq:def:Decoupling2}). 
\end{proof}

We can now formulate and prove the following theorem, which shows convergence between the $H_n^k$ and the $\lap(h_n^k)$ defined in Sections~\ref{sec:Decoupling} and~\ref{sec:Decoupling_continuous} respectively. In doing so, we will also prove Lemma~\ref{lemma:contDecoupling}. Since we will be using Thm.~\ref{thm:PHF_C} to do so, which in turn utilizes the former, it is worth taking a moment to make sure that in each induction step in the proof of Thm.~\ref{thm:Convergence} for some fixed $n+1$, we use the statement of Thm.~\ref{thm:PHF_C} for $n$, and then proceed to prove Lemma~\ref{lemma:contDecoupling} for $n+1$. We therefore do not enter any circular reasoning.  

\begin{theorem}\label{thm:Convergence}
Let $\pi/\theta\in\mathbb{Z}$ and $H_n^k$, $\lap(h_n^k)$ be defined by (\ref{eq:buildPHF}),~(\ref{eq:buildPHF_C}) respectively. Then \begin{align}\label{eq:convergence_thm_1}
    \lim_{\mu\to 0}\mu^{k\pi/\theta+2n}H_n^k\left(e^{-\mu x},e^{-\mu y}\right)=\alpha_{n,k}\lap\left(h_n^k\right)(x,y)
\end{align}
for some constants $\alpha_{n,k}\neq 0$.\\
Furthermore, we can write \begin{align}\label{eq:convergence_thm_2}
    H_n^k(x,y)=\frac{p_n^k(x,y)}{(1-x)^u(1-y)^v},\quad \lap(h_n^k)(x,y)=\frac{q_n^k(x,y)}{x^\alpha y^\alpha},
\end{align}
where $u,v,\alpha\in\Z$ with $u,v\leq\alpha=k\pi/\theta+2(n-1)$, $p_n^k(x,y)$ a polynomial and $q_n^k(x,y)$ a homogeneous polynomial of degree $\pi/\theta+2(n-1)$. 
\end{theorem}
\begin{proof}[Proof of Thm.~\ref{thm:Convergence} and Lemma~\ref{lemma:contDecoupling}]
We prove the theorem and the lemma simultaneously by induction. For $n=1$ everything can be checked by a direct computation. Now suppose everything is shown up to some $n$. By Lemma~\ref{lemma:limit_Decoupling}, we know that we can define a continuous decoupling function $f_n^k(x)$ of $\lap(h_n^k)(x,y)$ via a the scaling limit $f_n^k(x):=\lim_{\mu\to 0}\mu^{k\pi/\theta}F_n^k(e^{-\mu x})$, and we also know that it is of the form $f(x)=\beta x^{-k\pi/\theta-2n}$, so in particular Lemma~\ref{lemma:contDecoupling} holds for $n+1$ as well. Having now completely proven Thm.~\ref{thm:PHF_C} for $n+1$ (where we utilized Lemma~\ref{lemma:contDecoupling}), we can take the piecewise limit of (\ref{eq:buildPHF}). Using in particular Lemma~\ref{lemma:convergence_toolkit1}, and the definition (\ref{eq:buildPHF_C}) of $\lap(h_{n+1}^k)(x,y)$, we have for $\alpha=k\pi/\theta+2n$ and (to save space) $e_x=e^{-\mu x}$, $e_y:=e^{-\mu y}$, \begin{align*}
    &\lim_{\mu\to 0}\mu^{\alpha+2}H_n^k\left(e_x,e_y\right)\\
    =&\lim_{\mu\to 0}\frac{\mu^{\alpha}\left(e_xe_yK(e_x,e_y)-F_n^k(e_x)-\left[X_+(e_y)e_yH_n^k\left(X_+(e_y),e_y\right)-F_n^k\left(X_+(e_y)\right)\right]\right)}{\mu^{-2}K(e^{-\mu x},e^{-\mu y})}\\
    =&\frac{\lap(h_n^k)(x,y)-f_n^k(x)-\left[\lap(h_n^k)(c_+y,y)-f_n^k(c_+y)\right]}{\gamma(x,y)}\\
    =&\lap(h_{n+1}^k)(x,y).
\end{align*} 
The degree of $q_n^k(x,y)$ and the value of $\alpha$ in (\ref{eq:convergence_thm_2}) can be checked by a direct computation (note that it is allowed that $q_n^k(x,y)$ be divisible by some power of $x$ or $y$). From there it follows immediately that $\alpha=k\pi/\theta+2n$ is an upper bound of $u,v$ using (\ref{eq:convergence_thm_1}).
\end{proof}

\section{Guessing a decoupling function}\label{sec:guessing}
In \cite{Poly}, the authors used an entirely different approach to find biharmonic functions, which may not be as easy to generalize as the method above, but is in many ways a more elementary and intuitive approach. As we will see, their guessing method using an ansatz can be shown to be effective for computation of biharmonic functions whenever decoupling is possible. Unlike the constructive approach above, this ansatz allows to rule out the existence of decoupling functions of a sufficiently nice shape.\\
If we substitute $x\mapsto X_\pm(y)$ into (\ref{eq:FE}), then we obtain (see \cite[(30)]{Poly})\begin{align}\label{eq:eq30poly}
K(X_+,0)H_2(X_+,0)-K(X_-,0)H_2(X_-,0)=y\left[X_+H_1(X_+,y)-X_-H_1(X_-,y)\right].
\end{align}
Our goal is to rewrite the right-hand side of (\ref{eq:eq30poly}) for $H_1(x,y)=\frac{P(\omega(x))-P(\omega(X_+))}{K(x,y)}$. Seeing as we substitute $x\mapsto X_\pm$, where the denominator is $0$, we need to utilize L'Hôpital's rule, which gives us (note that $K(x,y)=\tilde{a}(y)\left(x-X_+(y)\right)\left(x-X_-(y)\right)$ \begin{align}
    X_+yH_1(X_+,y)=\frac{yX_+(y)\omega'(X_+)}{\tilde{a}(y)\left(X_+-X_-\right)}P'(\omega(X_+)),\\ X_-yH_1(X_-,y)=\frac{yX_-\omega'(X_-)}{\tilde{a}(y)\left(X_--X_+\right)}P'(\omega(X_-)).
\end{align}
Noting that $P'(\omega(X_+))=P'(\omega(X_-))$ due to the invariance property of $\omega$, in order to find a decoupling function $F(x)$ it would be enough to find an $F$ such that \begin{align}
    F(X_+)-F(X_-)=\frac{yX_+(y)\omega'(X_+)}{\tilde{a}(y)\left(X_+-X_-\right)}-\frac{yX_-\omega'(X_-)}{\tilde{a}(y)\left(X_--X_+\right)}.
\end{align}
To do so, we utilize a parametrization of the zero set of the kernel.

\subsection{A parametrization of the kernel curve}\label{sec:parametrization}
The curve $\mathcal{C}:=\{(x,y)\in\bar{\C}^2:K(x,y)=0\}$ permits a parametrization of the form \begin{align}
    x(s)&=\frac{(s-s_1)(s-1/s_1)}{(s-s_0)(s-1/s_0)},\\
    y(s)&=\frac{(\rho s-s_3)(\rho s-1/s_3)}{\rho s-s_2)(\rho s-1/s_2)},
\end{align}
where $\rho=e^{-i\theta}$, \begin{align}\label{eq:guessing_parametrization_s0}
    s_0&=\frac{2-(x_1+x_4)-2\sqrt{(1-x_1)(1-x_4)}}{x_4-x_1},\\
    s_1&=\frac{x_1+x_4-2x_1x_4-2\sqrt{x_1x_4(1-x_1)(1-x_4)}}{x_4-x_1},\label{eq:guessing_parametrization_s1}
\end{align}
with similar definitions for $s_2$ and $s_3$ using $y_1,y_4$ instead of $x_1,x_4$, where the $x_i$ and $y_i$ are defined by the zeros of the discriminant of the kernel as in Section~\ref{sec:Prelims} (see \cite{Book,MBM} for details). Using this parametrization, we have \cite[2.3]{FR11} \begin{align}
    x\left(\frac{1}{s}\right)=x(s),\quad  y\left(\frac{q}{s}\right)=y(s),
\end{align}
with $q:=e^{2i\theta}=1/\rho^2$. One can deduce that the mappings $s\mapsto \frac{1}{s}, s\mapsto \frac{q}{s}$ correspond to the restriction of the group to $\mathcal{C}$, and due to the invariance properties of $x(s),y(s)$ and $\omega(x)$, we see that $\omega(x(s))=s^{\pi/\theta}+s^{-\pi/\theta}+c$, for $c$ some some constant. In the following, seeing as with $\omega(x)$, we know that $\omega(x)-c$ will also be a suitable conformal mapping for our purposes, we will assume that $\omega(x(s))=s^{\pi/\theta}-s^{-\pi/\theta}$. \\
Using these parametrizations, one eventually finds that the right-hand side of (\ref{eq:eq30poly}) written in terms of $s$ takes the form \begin{align}\label{eq:guessing_2}
    c\left(s^{\pi/\theta}-s^{-\pi/\theta}\right)\frac{Q(s)}{s^2(s-1)(s+1)(s-q)(s+q)}=:B(s),
\end{align}
where $c\in\C$ is a constant and $Q(s)$ is a polynomial of degree $8$. Noticing that the mapping $s\mapsto \frac{s}{q}$ maps $X_+$ to $X_-$, we want to find a function $F$ such that \begin{align}\label{eq:guessing_1}
    f(s)-f(s/q)=B(s).
\end{align}
Then, one would only need to find a way to write $f(s)$ -- which must inherit the invariance property $f(s)=f(1/s)$ from $x(s)$ -- to a function of the form $F(x(s))$, and we would have our decoupling function. 

\subsubsection*{Rewriting the boundary value problem}

The following computations were originally done in \cite{DecouplingEq}. Writing \\$H(x,y)=\frac{P(\omega(x))-P(\omega(X_+))}{K(x,y)}$ we have, using L'Hôpital's rule, \begin{multline}
X_+yH(X_+,y)-X_-yH(X_-,y)\\=\frac{yX_+(y)\omega'(y)}{\tilde{a}(y)\left[X_+(y)-X_-(y)\right]}P'(\omega(X_+(y)))-\frac{yX_-(y)\omega'(y)}{\tilde{a}(y)\left[X_-(y)-X_+(y)\right]}P'(\omega_-(y))),
\end{multline}
where $\tilde{a}(y)$ is defined as in Section~\ref{sec:Prelims}. As $\omega(X_+(y))=\omega(X_-(y))$, it therefore suffices to find a decoupling function $F(x)$ such that \begin{align}\label{eq:ansatz_rewriting_1}
    F(X_+)-F(X_-)=\frac{yX_+(y)\omega'(X_+(y))}{\tilde{a}(y)\left[X_+(y)-X_-(y)\right]}-\frac{yX_-(y)\omega'(X_-(y))}{\tilde{a}(y)\left[X_-(y)-X_+(y)\right]}.
\end{align}
Using the parametrization, we have $X_+(y)=x(s)$, $X_-(y)=x(q/s)$. (\ref{eq:ansatz_rewriting_1}) thus becomes
\begin{align}\label{eq:ansatz_rewriting_2}
    f(x(s))-f(x(q/s))=\frac{y(s)x(s)\omega'(x(s))}{\tilde{a}(y(s))\left[x(s)-x(q/s)\right]}-\frac{y(s)x(s)\omega'(x(q/s))}{\tilde{a}(y(s))\left[x(q/s)-x(s)\right]}.
\end{align}
In order to simplify the right-hand side of (\ref{eq:ansatz_rewriting_2}), the main idea is to utilize the fact that $\omega(x(s))=s^{\pi/\theta}+s^{-\pi/\theta}$, thus \begin{align}
    \omega'(x(s))x'(s)=\frac{\pi}{\theta}\frac{1}{s}\left(s^{\pi/\theta}-s^{-\pi/\theta}\right).
\end{align}
We therefore rewrite \begin{align}
 \frac{y(s)x(s)\omega'(x(s))}{\tilde{a}(y(s))\left[x(s)-x(q/s)\right]}=\underbrace{\frac{y(s)}{\tilde{a}(y(s))\left[x(s)-x(q/s)\right]}}_{:=T_1(s)}\underbrace{\frac{x(s)}{x'(s)}}_{:=T_2(s)}\underbrace{\omega'(x(s))x'(s)}_{:=T_3(s)}.
\end{align}

By utilizing the fact that \begin{align}
    T_1(s)=\frac{y(s)}{\tilde{a}(y(s))\left[x(s)-x(q/s)\right]}=\frac{y(s)}{\tilde{D}(y(s))},
\end{align}
with $\tilde{D}$ the determinant from Section~\ref{sec:Prelims}, after some computations one obtains that, for some constant $c_1$, \begin{align}
    T_1(s)=c_1\frac{(\rho s_3 s-1)(\rho s -s_3)(\rho s_2 s-1)(\rho s-s_2)}{s(\rho s-1)(\rho s+1)}.
\end{align}
Similarly, we have \begin{align}
    T_2(s)=c_2\frac{(s_0 s-1)(s-s_0)(s_1 s-1)(s-s_1)}{(s+1)(s-1)}
\end{align}
for some constant $c_3$\footnote{This is the only part where our assumption $y_4\neq\infty$ comes into play; if $y_4=\infty$ we find that the denominator of $T_1(s)$ stays the same while the numerator is quadratic.}. We can therefore, after again some short computations, rewrite \begin{multline}
    \frac{y(s)x(s)\omega'(x(s))}{\tilde{a}(y(s))\left[x(s)-x(q/s)\right]}-\frac{y(s)x(s)\omega'(x(q/s))}{\tilde{a}(y(s))\left[x(q/s)-x(s)\right]}\\
    =\frac{y(s)}{\tilde{a}(y(s))\left[x(s)-x(q/s)\right]}\frac{\pi}{\theta}\frac{1}{s}\left(s^{\pi/\theta}-s^{-\pi/\theta}\right)\left(\frac{x(s)}{x'(s)}-\frac{s^2}{q}\frac{x(q/s)}{x'(q/s)}\right).
\end{multline}
After simplifying the last factor, we end up with
\begin{footnotesize}\begin{multline}\label{eq:guessing_rewriting_3}
   \frac{y(s)x(s)\omega'(x(s))}{\tilde{a}(y(s))\left[x(s)-x(q/s)\right]}-\frac{y(s)x(s)\omega'(x(q/s))}{\tilde{a}(y(s))\left[x(q/s)-x(s)\right]}\\
    =c\left(s^{\pi/\theta}-s^{-\pi/\theta}\right)\frac{(\rho s_3 s-1)(\rho s -s_3)(\rho s_2 s-1)(\rho s-s_2)(s-s_4)(s-q/s_4)(s-s_5)(s-q/s_5)}{s^2(s+1)(s-1)(s-q)(s+q)},
\end{multline}
\end{footnotesize}\noindent
where $c,s_4,s_5$ are constants. 
\subsection{The ansatz}
The guessing method used by the authors of \cite{Poly} is to search for a $f(s)$ of the form \begin{align}\label{eq:guessing_ansatz}
    f(s)=c\left(s^{\pi/\theta}-s^{-\pi/\theta}\right)\frac{s^pR(s)}{s-s^{-1}},
\end{align}
where $p$ is some constant and $R(s)$ rational. Utilizing (\ref{eq:guessing_rewriting_3}) and the fact that (due to the invariance property of $x(s)$) we must have $f(s)=f\left(1/s\right)$, it turns out that this already implies $p=-3$, and that $R(s)$ must be a reciprocal polynomial of degree $6$\footnote{In the case $y_4=\infty$, one finds $p=-1$ and $R(s)$ of degree $2$.}. From here, one can simply write $R(s)=(s-z_1)(s-\frac{1}{z_1})\dots (s-\frac{1}{z_3})$, and check to see if it is possible to find $z_1,z_2,z_3$ such that (\ref{eq:guessing_1}) holds. After some more calculations one finds that we want to solve \begin{multline}
    \label{eq:decoupling_ansatz_finaleq}
    (s - q) (s + q) R(s) - s^6q^{-2}(s - 1) (s + 1) R\left(q/s\right)=\\c(\rho s_3 s-1)(\rho s -s_3)(\rho s_2 s-1)(\rho s-s_2)(s-s_4)(s-q/s_4)(s-s_5)(s-q/s_5)
\end{multline}
for some constant $c$.
In (\ref{eq:guessing_2}), everything except for $s^{\pm\pi/\theta}$ is rational in $s$, and $s^{\pi/\theta}$ is invariant under $s\mapsto s/q$ (by definition, we have $q^{\pi/\theta}=e^{2i\theta\cdot \pi/\theta}=1$). Therefore it is not very surprising that we can, in a sense, leave the invariant factor $\left(s^{\pi/\theta}-s^{-\pi/\theta}\right)$ alone and find a decoupling function for the remaining part only. This is formalized by the following lemma. 
\begin{lemma}\label{lemma:guessing}
Let $t:=s^{\pi/\theta}$, and let $f_1\in\mathbb{C}(s,t)$ such that $f_1\in\C(s,t)$ is a decoupling function of $b(s)=R(s)h(t)$. Then we can find $f\in\mathbb{C}(s)$ such that $f$ is a decoupling function of $R(s)$.
\end{lemma}
\begin{proof}
Let $\tau$ be the automorphism $s\mapsto s/q$, which by definition of $q$ leaves $t$ invariant. We have two cases here, depending on whether or not $\pi/\theta$ is rational or not, i.e. if there is some algebraic relation between $s$ and $t$.
\begin{enumerate}
    \item \underline{$\pi/\theta\notin\Q$:}\newline
In this case, as we are always working with rational functions, we can utilize the fact that $s$ and $t$ are algebraically independent. By invariance of $t$ under $\tau$, we can thus write \begin{align}\label{eq:guessing_3}
    \frac{f_1(\tau s,t)}{h(t)}-\frac{f_1(s,t)}{h(t)}=R(s).
\end{align}
Due to the independence of $s$ and $t$ we can treat the right-hand side, viewed as a rational function coefficients in $\C(s)$, like a constant function in $t$.\\
Now let $f_2(s,t):=\frac{f_1(s,t)}{h(t)}=\frac{u(s,t)}{v(s,t)}$, with $u,v\in\mathbb{C}(s)[t]$. We factor $u,v$ into their irreducible components $(u_i),(v_i)$ over $\mathbb{C}(s)$. Any $u_i,v_i$ which lies in $\C[t]$ must cancel, because else it would be a factor of the entire left-hand side and thus of $R(s)$, a contradiction. Thus we can assume that all $u_i, v_i$ lie in $\C(s)[t]\setminus\C[t]$.\\
Suppose there is an $i$, which we can assume to be $1$, such that $\deg_t(v_i)>0$. Then, as the resulting pole must cancel, we know there is a $j$, let us say $j=2$, such that $v_2(s,t)=v_1(\tau(s),t)$. Proceeding inductively, we can construct a sequence $(v_i)$ such that $v_1(\tau^{n-1}(s),t)=v_n(s,t)$. But as $\tau^n\neq \operatorname{id}$ for all $n\in\N$ since $q\notin\Q$, this procedure will never stop, i.e. we would need to have an infinite number of factors $v_i$, which is impossible. Therefore, we know that $v(s,t)=v(s)\in\C(s)$; and by looking at the degree of the left-hand side of (\ref{eq:guessing_3}), therefore $u(s,t)=u(s)\in\C(s)$ as well. Hence, $f_2=\frac{f_1}{h}\in\C(s)$, and therefore it is a valid decoupling function of $R(s)$.
\item \underline{$\pi/\theta\in\Q$:}\\
Let $\pi/\theta=\frac{m}{n}$, with $m,n\in\N$, $(m,n)=1$. While before we could simply consider the irreducible factors of numerator and denominator of $f_2$, this is not now so simple anymore, as the algebraic structure of $\mathbb{F}:=\operatorname{Quot}\left(\bigslant{\mathbb{C}[s,t]}{\langle s^m-t^n\rangle}\right)$ is not as obvious. It would, for instance, not immediately make sense to talk about the degree of an expression. Therefore, we need to work around this issue. We know that $\operatorname{gcd}(m,n)=1$, and thus there are $a,b$ such that $am+bn=1$. Now consider the mapping\footnote{This is merely a formalisation of the basic idea of adjoining a simple element $T$ such that $T^n=s, T^m=t$, which serves the role of an $n$-th root of $s$.} \begin{align}
    \phi: \quad \mathbb{C}(T)\to \mathbb{F}: \quad T\mapsto s^bt^a.
\end{align}
One can check that an inverse map $\phi^{-1}$ is given by $(s,t)\mapsto (T^n,T^m)$, and that $\phi$ is in fact an isomorphism. Our automorphism $\tau$ can now be carried over to $\mathbb{C}(T)$ thus we obtain an automorphism $\sigma:=\phi^{-1}\circ\tau$. As $\tau$ and $\phi$ both fix $\mathbb{C}$, so does $\sigma$. Therefore, we know that $\sigma$ is of the form \begin{align}
    \sigma:\quad T\mapsto \frac{c_1 T+c_2}{d_2T+d_2}
\end{align}
for some $c_1,c_2,d_1,d_2$. From this, one obtains that $\sigma(T)=\zeta T$, where $\zeta^n=q, \zeta^m=1$, which means $\zeta=\exp \left( 2i\pi \frac{1}{m}\right)$.\\
Translating the decoupling property (\ref{eq:guessing_3}) to $T$, we now have, with $f_3=\phi^{-1}\circ f_2$ \begin{align}
    f_3(\sigma T)-f_3(T)=R_1(T).
\end{align}
By Hilbert's Thm. 90 \cite[4.6.5]{Book} such a $f_3$ can exist if and only if $\operatorname{Tr}_\sigma R_1=0$, that is, if \begin{align}
    \sum_{k=1}^m R_1(\zeta^k T)=0.
\end{align}
But using that $R_1(T):=\phi^{-1}\left(R(s)\right)=R(T^n)$ and that $\zeta^n=q$, we have \begin{align}
    \operatorname{Tr}_\sigma R_1=\sum_{k=1}^m R\left(\zeta^{kn}T^n\right)=\sum_{k=1}^m R\left(\phi^{-1}(q^k s)\right)=\phi^{-1}\left[\sum_{k=1}^m R(q^ks)\right]=\phi^{-1}\left(\operatorname{Tr}_\tau R\right).
\end{align}
As $\phi$ leaves $\mathbb{C}$ fixed, we therefore know that $\operatorname{Tr}_\sigma R_1=0 \Leftrightarrow \operatorname{Tr}_\tau R=0$, and thus a decoupling of $R_1(T)$ is possible if and only if a decoupling of $R(s)$ is possible. 
\end{enumerate}
\end{proof}
Lemma~\ref{lemma:guessing} implies that if the term $B(s)$ in (\ref{eq:guessing_2}) has a decoupling function in $\C\left(s,s^{\pi/\theta}\right)$, which is the most natural space to search a solution, then this decoupling function will already as a factor decouple the part not containing $s^{\pi/\theta}$. This is precisely the decoupling function we were searching for in the ansatz, which must therefore work.\vskip
\baselineskip
\textbf{Remarks:} \begin{itemize}
    \item This allows to show for concrete models with infinite group that a decoupling function in $\C\left(x,y,\omega(x),\omega(y)\right)$ cannot exist. However, it is not obvious how to show that an infinite group implies that there is no decoupling function. In particular, it is not at all clear how the group (as defined in Section~\ref{sec:Decoupling}, i.e. the birational transformations) being finite can be grasped in terms of the parametrization, where all that is left is the restriction of the group to the curve $\mathcal{C}$. Section~\ref{sec:pitheta2} will contain some examples where the restriction of the group, but not the group itself is finite.
    \item While it would be natural to assume that any decoupling function would be in $\C\left(s,s^{\pi/\theta}\right)$, this is not always the case. We will see an example for this in Example~\ref{sec:pitheta2_infinite_example}, where the ansatz will not work but we will construct a (non-algebraic) decoupling function. In the setting of Section~\ref{sec:Decoupling}, however, one checks immediately by (\ref{eq:DecouplingFormula}) that if ever the group is finite, then the resulting decoupling functions will be rational in $s$ and $s^{\pi/\theta}$. 
\end{itemize}

\subsubsection{Example: a decoupling function for the tandem walk using the ansatz}
Recall that the tandem walk was defined by the step set $\mathcal{S}=\{\nwarrow,\rightarrow,\downarrow\}$ with weight $\frac{1}{3}$ each. We had $K(x,y)=\frac{1}{3}\left(x^2+y+xy^2\right)-xy$, from which we obtain $x_1=0,x_4=4, y_0=\frac{1}{4},y_4=\infty$, and thus $s_0=-\frac{1}{2}(1+i\sqrt{3}),s_1=1,s_2=-1,s_3=\frac{1}{2}(1-i\sqrt{3})$. Furthermore, we have $\pi/\theta=3$, and thus $q:=e^{i2\theta}=-\frac{1}{2}(1-i\sqrt{3})$ as well as $\rho:=e^{-i\theta}=\frac{1}{2}(1-\sqrt{3})=-q$. This leads to \begin{align}
    x(s)=\frac{(s-1)^2}{1+s+s^2},\quad y(s)=\frac{-\frac{1}{2}(1-i\sqrt{3})s^2-\frac{1}{2}(1+i\sqrt{3})s+1}{\left(s+\frac{1}{2}(1+i\sqrt{3}\right)^2}.
\end{align}
After some computation, we see that the right-hand side of (\ref{eq:decoupling_ansatz_finaleq}) takes the form \begin{align}
    1/6 i (s-1)^3 (2 \sqrt{3} + (\sqrt{3}+3i) s) (1 + s + s^2)^2.
\end{align}
Letting $R(s)=(1-z_1)(1-z_2)(1-z_3)(1-1/z_1)(1-1/z_2)(1-1/z_3)$ and solving (\ref{eq:decoupling_ansatz_finaleq}) for $z_1,z_2,z_3$ yields multiple solutions, for instance $c=-\frac{1}{3}, z_1=1, z_2=i, z_3=-\frac{1}{2}(1+i\sqrt{3})$. This particular one leads to \begin{multline}
f(s):=c\left(s^3-s^{-3}\right)\frac{R(s)}{s^4-s^2}\\=\frac{1-s+2s^2-3s^3+3s^4-4s^5+3s^6-3s^7+2s^8-s^9+s^{10}}{3s^5}.
\end{multline}
To transform this back into a function of $x$, we utilize that \begin{align}
    \frac{1}{1-x(s)}=\frac{1+s+s^2}{3s}.
\end{align}
Making an ansatz of $f(s)=a_0+a_1\frac{1}{1-x(s)}+\dots+a_5\frac{1}{(1-x(s))^5}$ gives $a_0=a_1=0,a_2=-18,a_3=99,a_4=-162,a_5=81$. Putting this together, we finally obtain the decoupling function \begin{align}
    f(x)=\frac{9x (2 + 5 x + 2 x^2)}{(1 - x)^5}.
\end{align}
Note that this decoupling function is different from the one computed in Example~\ref{sec:exampleTandem}, where we obtained $F(x)=-\frac{81x^3}{4(1-x)^5}$. 

\subsubsection{Example: trying to decouple a model with infinite group}\label{sec:ansatz_infiniteexample}
Consider the model with the step set \begin{align}
    p_{1,0}=p_{0,1}=0,\\
 p_{1,1}=1/4,\\
 p_{1,-1}=p_{0,-1}=p_{-1,0}=p_{-1,1}=1/6,\\
 p_{-1,-1}=1/12.
\end{align}
This model has an infinite group, as can be seen by Thm.~\ref{thm:pitheta2:groupfinite}. We have $K(x,y)=\frac{1}{12}(1+2x+2y+2x^2+2y^2+3x^2y^2)-xy$, and can compute $x_{1,4}=y_{1,4}=\frac{1}{6}(\pm5\sqrt{3}-9)$. As $\pi/\theta=2$, we have $q=-1,\rho=i$. After some calculations, one finds that the right-hand side of (\ref{eq:decoupling_ansatz_finaleq}) takes, up to a multiplicative constant, the form \begin{align}
    s(\sqrt{2} - \sqrt{3} + i s) (\sqrt{3} + i s) (\sqrt{3} + 3 i s)  (1 + 
   i(\sqrt{2} - \sqrt{3}) s) (-1 + s^2).
\end{align} 
We can check that solving (\ref{eq:decoupling_ansatz_finaleq}) does not give any solutions for $z_1,z_2,z_3$, thus we cannot find a decoupling function using the ansatz. In particular, in light of Lemma~\ref{lemma:guessing} this implies that there is no rational (in $s$) decoupling function. 
We will, however, see how one can construct a non-algebraic decoupling function for this model using a contour integral in Example~\ref{sec:pitheta2_infinite_example}.

\section{The special case \texorpdfstring{$\pi/\theta=2$}{}}\label{sec:pitheta2}
In this section, we will consider the special case where $\theta=\pi/2$ ($\theta$ is the angle defined in Sec.~\ref{sec:Prelims}, and can be computed by (\ref{eq:defTheta})). This is done for two different reasons: firstly, we will see that in this setting we gain a number of nice properties; in particular for the group being finite (Thm.~\ref{thm:pitheta2:groupfinite}), in which case we the polyharmonic functions can be explicitly computed (Thm.~\ref{thm:pitheta2:finitegroup_PHF}). On the other hand, this setting allows us to explicitly compute decoupling functions in the case of an infinite group with comparably little effort, which are -- albeit not rational -- still guaranteed to exist by general theory about complex boundary value problems, see for example \cite[§4]{Gakhov},\cite[5.]{Book}. This will be done starting from Section~\ref{sec:pitheta2_infinitegroupcase}. In particular, we will see that the resulting functions are not even algebraic anymore (though still D-finite).\\
The case $\pi/\theta=2$ includes a number of standard models, such as the simple walk, the king's walk or the diagonal walk. It is characterized by the property \begin{align}\label{eq:pitheta2:condition_pitheta=2}
p_{1,1}+p_{-1,-1}=p_{1,-1}+p_{-1,1},
\end{align}
that is, the sums of the weights of the two diagonals are the same. This is a direct consequence of (\ref{eq:defTheta}). Also note that that the function $f(i,j)=ij$ satisfies $\triangle f(i,j)=0$ in this case, which corresponds to the fact that, up to a multiplicative constant, $H_1^1(x,y)=\frac{1}{(1-x)^2(1-y)^2}$. This can be checked directly using (\ref{eq:intro_1}). 

\subsection{A criterion for the group to be finite}
Deciding if the group of a given model is finite is in general not an easy problem, as can be seen for example by the very computationally heavy approach in \cite{KY}, or by the approach in \cite{MBM} where it was done for unweighted small-step models using a combination of eigenvalue properties and valuations. While for the standard models with $\pi/\theta=2$ like the simple walk, the diagonal walk or the king's walk the group is finite, this is not always the case\footnote{The fact that $\pi/\theta=2\in\Z$ guarantees that the \textit{restriction of the group on} $\mathcal{C}$ is finite, see e.g. \cite{FR11}, but not that it is finite on all of $\C^2$.}. Take for example the model with probabilities \begin{align}
 p_{1,0}=p_{0,1}=0,\\
 p_{1,1}=1/4,\\
 p_{1,-1}=p_{0,-1}=p_{0,-1}=p_{-1,1}=1/6,\\
 p_{-1,-1}=1/12.
\end{align}
One can check immediately that while we have $\pi/\theta=2$ and the restriction of the group has order $4$, the group itself is infinite.\\
In the case of $\pi/\theta=2$ we will show in this section that there is a very intuitive way to classify the behaviour of the group of a given model: it is finite of order $4$ if the model has either a North-South or an East-West symmetry, else it is infinite. To do so, we will first show that the group is of order four precisely if one of these symmetries holds; and then afterwards show that the group being finite leads directly back to this case. We will start with a technical lemma in order to shorten later computations. 

\begin{lemma}\label{lemma:pitheta2_onesymmetryallsymmetry}
Suppose we have a non-degenerate model with small steps, zero drift and $\pi/\theta=2$. If there is an $i\in\{-1,0,1\}$ such that $p_{1,i}=p_{-1,i}$, then we have $p_{1,j}=p_{-1,j}$ for $j\in\{-1,0,1\}$. Similarly, if there is an $i\in\{-1,0,1\}$ such that $p_{i,1}=p_{i,-1}$, then we have $p_{j,1}=p_{j,-1}$ for $j\in\{-1,0,1\}$.
\end{lemma}
\begin{proof}
By a direct computation, using (\ref{eq:pitheta2:condition_pitheta=2}). 
\end{proof}
This lemma tells us that we have a North-South symmetry in only one of the three possible ways, then our step set as a whole already has a North-South symmetry. We can now utilize this in the following

\begin{lemma}\label{lemma:pitheta2_4_symmetry}
Suppose we have a non-degenerate model with small steps, zero drift and $\pi/\theta=2$. Then, the group is finite of order $4$ precisely if the model has either a North-South, or an East-West symmetry. 
\end{lemma}
\begin{proof}
We know from \cite[Prop.~4]{Kilian_JEMS} that the order of the group is four if and only if the determinant \begin{align}
    \left|\begin{pmatrix}p_{-1,1}&p_{0,1}&p_{1,1}\\p_{-1,0}&-\frac{1}{t}&p_{1,0}\\p_{-1,-1}&p_{0,-1}&p_{1,-1}\end{pmatrix}\right|=0.
\end{align}
This determinant can be explicitly computed to be \begin{multline}
    -p_{1,0} p_{-1,1} p_{0,-1} + p_{1,0} p_{0,1} p_{-1,-1} + p_{1,1} p_{0,-1} p_{-1,0} \\- p_{0,1} p_{1,-1} p_{-1,0} -\frac{ p_{-1,1}p_{1,-1}+p_{1,1}p_{-1,-1}}{t}.
\end{multline}
Utilizing the fact that our walk has drift $0$, small steps, and (\ref{eq:pitheta2:condition_pitheta=2}), this can be simplified to \begin{align}
    \frac{t-1}{t}\left[(p_{-1,1}-p_{-1,-1}) (p_{-1,-1}-p_{1,-1})\right].
\end{align}
By Lemma~\ref{lemma:pitheta2_onesymmetryallsymmetry}, the last expression is $0$ precisely if our model has either a North-South, or an East-West symmetry.
\end{proof}

\begin{theorem}\label{thm:pitheta2:groupfinite}
Suppose we have a non-degenerate model with small steps, zero drift and $\pi/\theta=2$. Then the group is finite of order $4$ if the model has a North-South or an East-West symmetry, and it is infinite otherwise.
\end{theorem}
\begin{proof}
Utilizing Lemma~\ref{lemma:pitheta2_4_symmetry}, all that remains to show is that in our setting any group that is finite must be of order $4$. To see this, define $\tilde{\delta},\tilde{\varepsilon},\delta,\varepsilon\in\bar{\R}$ such that \begin{align}
    \frac{c(x)}{a(x)}&\stackrel{x\to\infty}{\to}\delta,&      \frac{c(x)}{a(x)}&\stackrel{x\to 0}{\to}\varepsilon,\\
        \frac{\tilde{c}(y)}{\tilde{a}(y)}&\stackrel{y\to \infty}{\to}\tilde{\delta},
        &\frac{\tilde{c}(y)}{\tilde{a}(y)}&\stackrel{y\to 0}{\to}\tilde{\varepsilon},
\end{align}
where $a(x),c(x),\tilde{a}(y),\tilde{c}(y)$ are defined as in Section~\ref{sec:Prelims}. Notice that if $\delta=\varepsilon=0$, then this would imply $p_{1,1}=p_{-1,1}=0$, in which case by Lemmas~\ref{lemma:pitheta2_onesymmetryallsymmetry} and \ref{lemma:pitheta2_4_symmetry} we already know the group to be finite of order $4$. In the same fashion, one sees that if $\delta=\varepsilon=\infty$, $\tilde{\delta}=\tilde{\varepsilon}\in\{0,\infty\}$ then we have a finite group of order $4$. In all other cases, we find that, for sufficiently large values of $(x,y)$, the group behaves like \begin{align}
    (x,y)\mapsto (\delta \bar{x},y)\mapsto (\delta \bar{x},\tilde{\varepsilon}\bar{y}) \mapsto \left(\frac{\varepsilon}{\delta}x,\tilde{\varepsilon}\bar{y}\right)\mapsto \left(\frac{\varepsilon}{\delta}x,\frac{\tilde{\delta}}{\tilde{\varepsilon}},y\right)\mapsto\dots
\end{align}
Consequently, for the group to be finite, both $\frac{\varepsilon}{\delta},\frac{\tilde{\delta}}{\tilde{\varepsilon}}$ must be roots of unity. As they are nonnegative reals, they must therefore be $1$. This condition can then be checked to simplify to $p_{1,1}p_{-1,-1}=p_{1,-1}p_{-1,1}$. \\
We write \begin{align}\label{eq:pitheta2:groupfinite:eq1}
    0=p_{1,1}p_{-1,-1}-p_{1,-1}p_{-1,1}&=p_{1,1}p_{-1,-1}-p_{1,-1}(p_{1,1}+p_{-1,-1}-p_{1,-1})\\
    &=(p_{1,1}-p_{1,-1})(p_{-1,-1}-p_{1,-1}),
\end{align}
where in (\ref{eq:pitheta2:groupfinite:eq1}) we made use of (\ref{eq:pitheta2:condition_pitheta=2}).
One can check by a short computation that the first condition implies a North-South and the second one an East-West symmetry; thus by Lemma~\ref{lemma:pitheta2_4_symmetry} we already know that our group is of order $4$. \end{proof}

\subsection{The finite group case}
We use the same parametrization of the curve $\mathcal{C}:=\{(x,y)\in\overline{\C}^2:K(x,y)=0\}$ as in Section~\ref{sec:parametrization}, \begin{align}
    x(s)&=\frac{(s-s_1)(s-1/s_1)}{(s-s_0)(s-1/s_0)},\\
    y(s)&=\frac{(\rho s-s_3)(\rho s-1/s_3)}{(\rho s-s_2)(\rho s-1/s_2)},
\end{align}
with $s_0,s_1,s_2,s_3$ given by (\ref{eq:guessing_parametrization_s0})--(\ref{eq:guessing_parametrization_s1}). Remember the invariance properties $x(s)=x\left(1/s\right)$ and $y(s)=\left(q/s\right)$ with $q:=e^{2i\theta}=1/\rho^2$, and that the mappings $s\mapsto \frac{1}{s}, s\mapsto \frac{q}{s}$ correspond to the restriction of the group to $\mathcal{C}$, and that we have $\omega(x(s))=s^{\pi/\theta}+s^{-\pi/\theta}$.\\
In the case $\pi/\theta=2$, the above immediately simplifies to $\rho=-i, q=-1$. By Thm.~\ref{thm:pitheta2:groupfinite}, we also know that in this case we have a North-South, or an East-West symmetry. This gives us some particularly nice properties of these models, and allows us to compute polyharmonic functions without the use of decoupling functions. 

\begin{lemma}\label{lemma:pitheta2_finitegroup_properties} 
Suppose we have a non-degenerate model with small steps, zero drift, $\pi/\theta=2$ and finite group. If we have an East-West symmetry, then there is a constant $c$ such that we can write \begin{align}\label{eq:pitheta2:w}
    \omega(x)=c\frac{x}{(1-x)^2},
\end{align}
and the contour $\mathcal{G}$ given by $X_\pm\left([y_1,y_4]\right)$ is the unit circle.\\
In case of a North-South symmetry, a corresponding statement holds true for $y$ and the corresponding conformal mapping $\widehat{\omega}$ instead. 
\end{lemma}
\begin{proof}
(\ref{eq:pitheta2:w}) follows from the parametrization, computing \begin{align}\label{eq:pitheta2:wshape}
\frac{x(s)}{(1-x(s))^2}=\frac{a_0s^4+a_1s^3+a_2s^2+a_1s+a_0}{s^2}.
\end{align}
It turns out that we have $a_1=0$ if $s_0s_1=0$, which we can check to be true if and only if we have a North-South symmetry. In this case we can then see that the right-hand side of (\ref{eq:pitheta2:wshape}) is $a_0\left(s^2+s^{-2}\right)+a_2$. As we already know that, up to an additive constant, $\omega(x(s))=s^2+s^{-2}$, we thus have $\omega(x)=\frac{1}{a_0}\frac{x}{(1-x)^2}-a_2$, and seeing as we pick $\omega(x)$ such that $\omega(0)=0$, (\ref{eq:pitheta2:w}) follows.\\
For the statement about the contour, \cite[Thm.~5.3.3]{Book} tells us that it is a circle if $\pi/\theta=2$ (in their notation, $r=0$). However, considering that $\omega(x)$ must take the same values on the upper and lower half of the contour $\mathcal{G}$ by the invariance property $\omega(X_+(y))=\omega(X_-(y))$ and given (\ref{eq:pitheta2:w}), it follows that this circle must be the unit circle. 
\end{proof}

\begin{theorem}\label{thm:pitheta2:finitegroup_PHF}
For any non-degenerate model with small steps, zero drift, $\pi/\theta=2$, finite group, and East-West symmetry, an explicit basis of polyharmonic functions is given by\begin{align}\label{eq:pitheta2:finitegroup:PHF_formula}
    H_n^k(x,y)=\beta^{n-1}\omega(y)^{n-1}\frac{y^{n-1}}{(1-x)^2(1-y)^{2n}}\left[\sum_{j=0}^{k-1}s_{n}(j)\omega\left(X_+\right)^j\omega(x)^{k-j-1}\right],
\end{align}
for some constant $\beta$, where $s_l: \mathbb{N}\to\mathbb{N}$ is defined inductively via $s_1(j)=1, s_{l+1}(j)=\sum_{i=1}^{j+1}s_l(j)$.\\
In particular, this basis has the form \begin{align}
    H_n^k(x,y)=\frac{p_{m,k}}{(1-x)^{2k}(1-y)^{2(m+k-1)}},
\end{align}
for $p_{m,k}$ some polynomial of degree at most $2k+n-3$.\par
In case of a North-South symmetry, the statement holds with $x$ and $y$ reversed.
\end{theorem}
\begin{proof}
To illustrate the idea, let us start computing $H_n^1(x,y)$. We know that, for some constant $\alpha$, we have \begin{align}
H_1^1(x,y)=\alpha\frac{1}{(1-x)^2(1-y)^2}.
\end{align}
Since the group is finite, we can without loss of generality assume that we have an East-West symmetry; otherwise we exchange the roles of $x,y$ in the following. In this case, by Lemma~\ref{lemma:pitheta2_finitegroup_properties}, we can rewrite \begin{align}
    xyH_1^1(x,y)= xy\frac{\omega(x)-\omega(X_+)}{K(x,y)}=\beta \omega(x)\omega(y),
\end{align} 
where again $\beta$ is some constant. As $\omega(X_+)=\omega(X_-)$, we see immediately that we do not need a decoupling function, and instead we can continue via \begin{align}
    H_2^1(x,y)&=\frac{xyH_1^1(x,y)-X_+yH_1^1(X_+,y)}{K(x,y)}\\
    &=\frac{\beta}{K(x,y)}\left[\omega(x)\omega(y)-\omega(X_+)\omega(y)\right]\\
    &=\beta\omega(y)\underbrace{\frac{\omega(x)-\omega(X_+)}{K(x,y)}}_{=H_1^1(x,y)}.
\end{align}
We can continue inductively, noticing that in each step we only gain a factor of $\beta\omega(y)$, and thus obtain \begin{align}
    H_n^1(x,y)&=\beta^{n-1}\omega(y)^{n-1}H_1^1(x,y)\\
    &=\frac{\alpha\tilde{\beta}^ny^{n-1}}{(1-x)^2(1-y)^{2n}}.
\end{align}
If now $k\geq 2$, then we can compute \begin{align}
    H_1^k&=\frac{\omega(x)^n-\omega(X_+)^n}{K(x,y)}\\
    &=\underbrace{\frac{\omega(x)-\omega(X_+)}{K(x,y)}}_{=H_1^1(x,y)}\left[\omega(x)^{n-1}+\dots+\omega(X_+)^{n-1}\right].
\end{align}
The computation now continues in exactly the same fashion as for $k=1$, except one needs to carry along more terms of the form $\omega(x)^a\omega(y)^b$, with their coefficients, which is where (\ref{eq:pitheta2:finitegroup:PHF_formula}) comes from. We have \begin{multline}
    xyH_n^k(x,y)-X_+yH_n^k(X_+,y)\\
    =\beta^m\omega(y)^m\left[\sum_{j=0}^{k-1}\omega(x)_{k-1}\omega(X_+)^j-\omega(X_+)^k\sum_{j=0}^{k-1}s_m(j)\right].
\end{multline} 
Using the algebraic identity \begin{align}
    a\left[\sum_{j=0}^{k-1}c_j a^{k-j-1}b^j\right]-b^n\sum_{j=0}^{k-1}c_j=(a-b)\sum_{j=0}^{k-1}\left(\sum_{i=1}^{j+1}c_i\right)a^{k-j-1}b^j
\end{align}
for $a=\omega(x)$ and $b=\omega\left(X_+\right)$ then yields the statement.
\end{proof}

\subsection{The infinite group case}\label{sec:pitheta2_infinitegroupcase}
In the case of an infinite group, the approach as in the previous section clearly does not work, as it was dependent on the fact that we could write, up to multiplicative constants, $xyH_1^1(x,y)=\frac{xy}{(1-x)^2(1-y)^2}=\omega(x)\omega(y)$. For this, in the finite group case we utilized the special shape of $\omega$ given by Lemma~\ref{lemma:pitheta2_finitegroup_properties}, which is now unavailable. Neither can we use (\ref{eq:DecouplingFormula}) to find a decoupling function and simplify our BVP. However, general theory of these boundary value problems as in \cite{Gakhov} still tells us that a decoupling function should exist, and there are methods to find them. It is therefore only natural to try and see what happens if we want to apply them here. Unfortunately, it will turn out that even in this simple case $\pi/\theta=2$, the resulting functions are rather unwieldy, and will in general not even be algebraic anymore.
\vskip\baselineskip
Suppose from now on that we have an arbitrary non-singular model with small steps, zero drift and $\pi/\theta=2$. We already know that $H_1^1(x,y)=\frac{1}{(x-1)^2(y-1)^2}$ is a harmonic function for such model, and what we want to do is to compute a biharmonic function of $H_1^1(x,y)$.\\
Now let \begin{align}
    Y_\pm(x)=\frac{-b(x)\mp \sqrt{b(x)^2-4a(x)c(x)}}{2a(x)},
\end{align}
as in Section~\ref{sec:Prelims}, and consider the contour $\Gamma$ given by $X_\pm[y_1,1]$. By \cite[Lemma~6.5.1]{Book}, we know that $\Gamma$ is a circle, symmetric with respect to the real axis, which it intersects at $1$ and at some point $-1<p$. We let $c,d$ be the center and radius of $\Gamma$ respectively. Let, again as in Section~\ref{sec:Prelims}, $\mathcal{G}$ be the (finite) domain bounded by $\Gamma$. Via \begin{align}
    r:\quad \mathbb{C}\to\bar{\mathbb{C}}:\quad z\mapsto d^2\frac{1}{z-c}+c,
\end{align}
we can define a rational mapping $r$ such that \begin{enumerate}
    \item $r$ is an involution,
    \item $r$ maps the interior $\mathcal{G}^\circ$ to the exterior $\mathcal{G}^c$ and vice versa,
    \item $r$ corresponds to complex conjugation on $\Gamma$ itself. 
\end{enumerate}
The existence of this rational mapping is the main reason why the following computation turns out to be comparatively simple; if $\pi/\theta$ were not $2$, then $\Gamma$ would not be a circle and things would end up being more complicated.\\
Now define \begin{align}\label{eq:pitheta2:shapeofL_1}
    L(x):=xY_+(x)H_1(x,Y_+(x))-r(x)Y_+(x)H_1(r(x),Y_+(x)).
\end{align}
$L(x)$ describes the value of $xyH_1^1(x,y)-X_+yH_1^1(X_+,y)$ on $\Gamma$: we substitute $Y_+(x)$ for $y$ to be on $\Gamma$ in the first place, and complex conjugation corresponds to switching from one solution of $K(x,Y_+(x))=0$ to the other, thus we have $K(x,Y_+(x))=K(r(x),Y_+(x))=0$. It is also the expression we want to find a decoupling function of; our goal is to find a $\Upsilon$, which is analytic inside $\mathcal{G}\setminus\{1\}$, such that \begin{align}\label{eq:pitheta2:decouplingcondition}
    \Upsilon(x)-\Upsilon(r(x))=L(x),\quad\forall x\in\Gamma.
\end{align}
\begin{lemma}\label{lemma:pitheta2_shapeofL}
We have \begin{align}
\label{eq:pitheta2_infinitegroup_expr_L}
    \alpha L(x)=\frac{\left[(-1+2d+x)\right]\left[p_2(x)+\sqrt{b(x)^2-4a(x)c(x)}\right]}{(x-1)^3}
\end{align}
for some non-zero constant $\alpha$ and a polynomial $p_2(x)$ of degree at most $2$.
\end{lemma}
\begin{proof}
We now utilize that we know the shape of $H_1(x,y)$, and rewrite (\ref{eq:pitheta2:shapeofL_1}) as\begin{align}
    \beta L(x)&=\left[\frac{x}{(x-1)^2}-\frac{r(x)}{(r(x)-1)^2}\right]\frac{Y_+}{(Y_+-1)^2}\\
    \label{eq:pitheta2:lemmashapeL:1}
    &=\frac{(d-1) (-1 + 2 d + x)}{d^2 (x-1)}\frac{Y_+}{(Y_+-1)^2},
\end{align}
where $\beta$ is some non-zero multiplicative constant we can ignore in the following. Note here that the factor $(-1+2d+x)$ has a zero at $x=1-2d=p$ (where $p$ was defined to be the second intersection, other than $1$, of $\mathcal{G}$ with the real axis), which will be important later on. In order to simplify $\frac{Y_+}{(Y_+-1)^2}$, seeing as $a(x)+b(x)+c(x)=K(x,1)=\rho_1(x-1)^2$ and $b(x)^2-4a(x)c(x)=(x-1)^2\tilde{p}(x)$ for a constant $\rho_1$ and a quadratic polynomial $\tilde{p}(x)$, we can see the following identities which we will use later on: \begin{align}
\label{eq:pitheta2_infinitegroup_identity1}
    a(1)=c(1)\\
    \label{eq:pitheta2_infinitegroup_identity2}
    a(1)+b(1)+c(1)=K(1,1)=0,\\
    \label{eq:pitheta2_infinitegroup_identity3}
    a'(1)+b'(1)+c'(1)=\left.\frac{\partial}{\partial x}K(x,1)\right|_{x=1},\\
    \label{eq:pitheta2_infinitegroup_identity4}
    \left.\frac{\partial}{\partial x}b(x)^2-4a(x)c(x)\right|_{x=1}=\left.\frac{\partial}{\partial x}\triangle(x)\right|_{x=1}=0.
\end{align} 
A direct simplification yields \begin{align}
    \frac{Y_+}{(Y_+-1)^2}&=\frac{\left[-b(x) + \sqrt{b(x)^2 - 4 a(x) c(x)}\right] \left[2 a(x) + b(x) + \sqrt{b(x)^2 - 4 a(x) c(x)}\right]^2}{8\left[a(x)+b(x)+c(x)\right]^2}\\
    &=\frac{\left[-b(x) + \sqrt{b(x)^2 - 4 a(x) c(x)}\right] \left[2 a(x) + b(x) + \sqrt{b(x)^2 - 4 a(x) c(x)}\right]^2}{8\rho_1^2a(x)(x-1)^4}.
\end{align}
In order to simplify the numerator, we write \begin{align}
    &\left[-b(x) + \sqrt{b(x)^2 - 4 a(x) c(x)}\right] \left[2 a(x) + b(x) + \sqrt{b(x)^2 - 4 a(x) c(x)}\right]^2\\
    \label{eq:pitheta2_infinitegroup_identity0}
    =&
    -4a(x)\left[a(x)b(x)+4a(x)c(x)+b(x)c(x)\right]+a(x)(a(x)-c(x))\sqrt{b(x)^2 - 4 a(x) c(x)}.
\end{align}
Note that the factor $a(x)-c(x)$ in front of the square root vanishes if and only if our walk has a North-South symmetry\footnote{In particular, had we picked a model with East-West, but no North-South symmetry here, then the computation here is much more complicated than necessary -- we could just have swapped the roles of $x$ and $y$, done the same calculations and at this point ended up with a purely rational expression.}, which is why everything stays rational in the finite group case.\\
We want to show that the first summand contains a factor of $(x-1)^2$. To do so, we notice that for $s_1:=a(x)b(x)+4a(x)c(x)+a(x)b(x)$, we have \begin{align}
    s_1(1)=&a(1)b(1)+4a(1)c(1)+b(1)c(1)\\
    &\stackrel{(\ref{eq:pitheta2_infinitegroup_identity1})}{=}2a(1)\left[b(1)+2a(1)\right]\\
    &\stackrel{(\ref{eq:pitheta2_infinitegroup_identity2})}{=}0,
\end{align}
as well as \begin{align}
    s_1'(1)&=a'(1)b(1)+b'(1)a(1)+4a'(1)c(1)+4a(1)c'(1)+b'(1)c(1)+b(1)c'(1)\\
    &=b(1)\left[a'(1)+c'(1)\right]+b'(1)\left[a(1)+c(1)\right]+4\left[a'(1)c(1)+a(1)c'(1)\right]\\
    &\stackrel{(\ref{eq:pitheta2_infinitegroup_identity3})}{=}
    -b(1)b'(1)-b'(1)b(1)+4\left[a'(1)c(1)+a(1)c'(1)\right]\\
    &\stackrel{(\ref{eq:pitheta2_infinitegroup_identity4})}{=}0.
\end{align}
Thus, $-4s_1(x)=(x-1)^2p_1(x)$ for some polynomial $p_1(x)$ of degree at most $2$. To do the same for the second summand of (\ref{eq:pitheta2_infinitegroup_identity0}), we notice that the fact that $a(1)-c(1)=0$ is nothing but (\ref{eq:pitheta2_infinitegroup_identity1}). To see that $a'(1)-c'(1)=0$, let $a(x)=a_0+a_1x+a_2x^2, c(x)=c_0+c_1x+c_2x^2$ (that is, $a_{0/1/2}=p_{-1/0/1,1}, c_{0/1/2}=p_{-1/0/1,-1}$). The expression $a'(1)-c'(1)$ thus simplifies to $2a_2+a_1-2c_2-c_1$. Utilizing that, as $\pi/\theta=2$, we have $a_2+c_0=c_2+a_0$, we can write the latter as $2(a_2-c_2)+a_1-c_1=a_2-c_2+a_0-c_0+a_1-c_1=(a_2+a_1+a_0)-(c_2+c_1+c_0)=0$, since we have zero drift. Therefore, we know that $a(1)-c(1)=a'(1)-c'(1)=0$, and therefore $a(x)-c(x)=\rho (x-1)^2$ (note that $a(x)-c(x)$ is quadratic in $x$). Thus we obtain \begin{align}
    \frac{Y_+}{(Y_+-1)^2}&=\frac{\left[-b(x) + \sqrt{b(x)^2 - 4 a(x) c(x)}\right] \left[2 a(x) + b(x) + \sqrt{b(x)^2 - 4 a(x) c(x)}\right]^2}{8\rho_1^2(x-1)^4}\\
    &=\frac{p_2(x)(x-1)^2+(x-1)^2\sqrt{b(x)^2-4a(x)c(x)}}{\rho_3(x-1)^4}\\
    &=\frac{p_2(x)+\sqrt{b(x)^2-4a(x)c(x)}}{\rho_3(x-1)^2}.
\end{align}
Substituting this into (\ref{eq:pitheta2:lemmashapeL:1}) yields the statement.\end{proof}
We now define \begin{footnotesize}\begin{align}\label{eq:pitheta2_infinite_L1}
    L_1(x)&:=\frac{(-1+2d+x)p_2(x)}{(x-1)^3},\\
    \label{eq:pitheta2_infinite_L2}
    L_2(x)&:=\frac{(-1+2d+x)\sqrt{b(x)^2-4a(x)c(x)}}{(x-1)^3}=\frac{(-1+2d+x)\sqrt{-(x-x_1)(x-x_4)}}{(x-1)^2}.
\end{align}
\end{footnotesize}\noindent
By construction, we have $\left[L_1(x)+L_2(x)\right]=\alpha L(x)$.\\
We would now like to proceed by computing decoupling functions of $L_1,L_2$ separately. For decoupling functions of $L_{1,2}(x)$ to exist, we must have $L_{1,2}(x)+L_{1,2}(r(x))=0$ for $x\in\Gamma$, due to (\ref{eq:pitheta2:decouplingcondition}). Note that $L(x)$ satisfies this condition by construction.\\
The first question to ask here is in which way we define the square root. This depends on the sign of $(x-x_1)(x-x_4)$; in order to utilize our methods later we will want the expression $L_2(x)$ to be continuous on the contour $\Gamma$. Due to \cite[Thm.~5.3.3]{Book}, we know that $x_1\in\mathcal{G}$, $x_4\in\mathcal{G}^c$. We select the branch cut such that the root singularity on $\Gamma$ is canceled out by the factor $(-1+2d+x)=(x-p)$, for $p$ the left intersection of $\Gamma$ with the real axis, i.e. we need to select the branch cut along the axis with the sign of $-(p-x_1)(p-x_4)$.\\
In both cases, there is a section of the contour $\Gamma$ which lies on the side of the branch cut. Therefore, on this section $\Gamma'$ of the contour we have \begin{align*}\sqrt{-(\bar{x}-x_1)(\bar{x}-x_4)}=-\overline{\sqrt{-(x-x_1)(x-x_4)}}.\end{align*} This implies that, on $\Gamma'$,\begin{align}
    L(x)+L(r(x))=L_1(x)+L_1(\overline{x})+L_2(x)+L_2(\overline{x})=\underbrace{L_1(x)+\overline{L_1(x)}}_{\in\R}+\underbrace{L_2(x)-\overline{L_2(x)}}_{\in\C}=0.
\end{align}
Consequently, we know that $L_1(x)+L_1(r(x))$ must be $0$ on $\Gamma'$, and as it is a rational function it must thus be $0$ everywhere. The same goes for $L_2(x)+L_2(r(x))$. 
This means that finding a decoupling function of $L(x)$ can be done in two parts:\begin{enumerate}
    \item We find a decoupling function of the rational function $L_1(x)$,
    \item We find a decoupling function of the non-rational function $L_2(x)$.
\end{enumerate}

\subsubsection*{Decoupling of the (rational) $L_1$}
As we already know that $L_1(x)+L_1(r(x))=0$, this turns out to be rather straightforward: we have \begin{align}
    L_1(x)=\frac{1}{2}\left[L_1(x)-L_1(r(x))\right],
\end{align}
which already gives us a rational decoupling function. One arguably gets a somewhat nicer form by utilizing an ansatz of the form \begin{align}
    L_1(x)=\alpha_3\left[\frac{1}{(x-1)^3}-\frac{1}{(r(x)-1)^3}\right]+\alpha_1\left[\frac{1}{x-1}-\frac{1}{r(x)-1}\right],
\end{align}
as will be done for Example~\ref{sec:pitheta2_infinite_example}.

\subsubsection*{Decoupling of the (irrational) $L_2$}\label{sec:pitheta2_infinite_decouplingnonrational}
Note that the previous approach is problematic here, as the resulting function would have singularities at $x=x_1$ and $x=x_4$, which might be inside of $\mathcal{G}$. Thus we compute a decoupling function via a contour integral, which is also the standard approach given the theory of complex boundary value problems. To utilize the theory as in e.g. \cite{Gakhov}, we need a function which is continuous on $\Gamma$. This is not the case for $L_2$, due to its pole at $x=1$. However, this can easily be remedied by considering instead of $L_2$ the function\begin{align}
    L_3(x):=(x-1)(r(x)-1)L_2(x)=-d\frac{(-1+2d+x)\sqrt{-(x-x_1)(x-x_4)}}{x-(1-d)}.
\end{align}
If we find a decoupling function $\Upsilon_3(x)$ of $L_3$, then $\Upsilon_2(x):=\frac{\Upsilon(x)_3}{(x-1)(r(x)-1)}$ will be a decoupling function of $L_2$, as the denominator is invariant under $x\mapsto r(x)$. However, due to \cite{Gakhov} we already know that such a $\Upsilon_3$ exists, seeing as $L_3(x)$ is continuous and bounded on $\Gamma$ (though not analytic near $x=p$, but this does not matter for us).\\
Since the theory guarantees us the existence of a $\Upsilon_3$, analytic in $\mathcal{G}^\circ$, which decouples $L_3$, we can utilise the same trick as in \cite{FR11}: we write the decoupling property, select a $t\in\mathcal{G}^\circ$, divide by $(x-t)$ and integrate over $\Gamma$ with respect to $x$. The resulting equation then reads \begin{align}\label{eq:pitheta2:irrationaldecoupling_1}
    \int_{\Gamma}\frac{\Upsilon_3(x)}{x-t}\mathrm{d}x-\int_{\Gamma}\frac{\Upsilon_3(r(x))}{x-t}\mathrm{d}x=\int_{\Gamma}\frac{L_3(x)}{x-t}\mathrm{d}x.
\end{align}
The leftmost term is, by Cauchy's integral formula, nothing but $2\pi i \Upsilon_3(t)$, and the rightmost term can be computed. The question is what to do with the middle term. We notice that $r(x)$ is an involution, sending $\Gamma$ to itself (only changing the direction along which $\Gamma$ is traversed), and that $r'(x)=-\frac{d^2}{(x-c)^2}$. Letting now $x=r(y)$, this integral can be written as \begin{align}
    \label{dcf_nonrational_eq1}
    \int_{\Gamma}\frac{\Upsilon_3(r(x))}{x-t}\mathrm{d}x&=d^2\int_\Gamma\frac{\Upsilon_3(y)}{r(y)-t}\frac{1}{(y-c)^2}\mathrm{d}y\\
    &=d^2\int_\Gamma\frac{\Upsilon_3(y)}{(r(y)-t)(y-c)}\frac{1}{y-c}\mathrm{d}y\\
    \label{dcf_nonrational_eq2}
    &=d^2\int_\Gamma\frac{\Upsilon_3(y)}{d^2+(c-t)(y-c)}\frac{1}{y-c}\mathrm{d}y.
\end{align}
From (\ref{dcf_nonrational_eq1}), we see that the only possible pole of the integrand is at $y=c$ (since $t\in\mathcal{G}^\circ$), and from (\ref{dcf_nonrational_eq2}) it follows that this is a simple pole. We can therefore, again by Cauchy, write \begin{align}
    \int_\Gamma\frac{\Upsilon_3(r(x))}{x-t}\mathrm{d}x=2\pi i \Upsilon(c).
\end{align}
Overall, we therefore get \begin{align}
    \Upsilon_3(t)-\Upsilon_3(c)=\frac{1}{2\pi i }\int_\Gamma \frac{L_3(x)}{x-t}\mathrm{d}x.
\end{align}
Noticing that with $\Upsilon_3$ any translation $\Upsilon_3+\operatorname{const}$ is also a solution for any constant, we can without loss of generality assume that $\Upsilon_3(c)=0$, and obtain \begin{align}
    \Upsilon_3(t)=\frac{1}{2\pi i}\int_\Gamma\frac{L_3(x)}{x-t}\mathrm{d}x.
\end{align} 
\textbf{Remark:} We can, unfortunately, not apply calculus of residues to the integral on the right-hand side of (\ref{eq:pitheta2:irrationaldecoupling_1}), since it is not analytic on the branch cut. We will take a closer look at this integral in the following section.

\subsubsection{Asymptotics of \texorpdfstring{$\Upsilon_3(t)$}{}}
The goal of this section is to compute the asymptotics of $\left[t^n\right]\Upsilon_3(t)$\footnote{Here, $\left[t^n\right]$ is the linear operator extracting the $n$-th coefficient of a power series around $0$.}, which will serve to show that $\Upsilon_3(t)$ cannot be an algebraic function. To do so, by standard methods about power series as presented for example in \cite{Flajolet}, we need to know the location of the singularity of $\Upsilon_3(t)$ closest to the origin.\\
Remember that the contour $\Gamma$ is a circle in $\C$, going through $\Gamma$ and intersecting the real axis at a second point $p$. For the asymptotics, we need a bit more information about the exact location of $p$. 

\begin{lemma}
We have
\begin{table}[ht]
\begin{tabular}{lll}
$|p|>1$ & if & $p_{1,1}>p_{1,-1},$ \\
$|p|<1$ & if & $p_{1,1}<p_{1,-1},$ \\
$|p|=1$ & if & $p_{1,1}=p_{1,-1}.$
\end{tabular}
\end{table}
\end{lemma}
\begin{proof}
Let us parametrize the contour and consider the absolute value of $X_0(t)X_1(t)$, $t$ close to $1$. Note again that we have \begin{align}
    f(t):=\left|X_0(t)\right|^2=X_0(t)X_1(t)=2\frac{c(t)}{a(t)}.
\end{align}
As we have zero drift and thus $X_0(1)=X_1(1)=0$, we see that $f'(1)=0$. Now all that remains to do is to check whether at this point we have a local minimum or maximum; i.e. whether we have $f''(1)>0$ or $f''(1)<0$ respectively. Therefore, we can use the explicit forms of $a(x),b(x)$, and in the end we obtain (using, again, that $r=0$) \begin{align}
    f''(1)=-4\frac{p_{1,-1}-p_{-1,-1}}{p_{-1,1}+p_{-1,0}+p_{-1,-1}}.
\end{align}
Since the denominator is certainly $>0$, and since (again, due to the drift being zero) we have $p_{1,-1}-p_{-1,-1}=p_{1,1}-p_{-1,1}$, the statement follows.\\
Note lastly that if we have equality, i.e. $p_{1,1}=p_{-1,1}$, then we already know that we have an East-West symmetry, and that in this case the contour will be the unit circle.
\end{proof}

In the case $|p|=1$, there is no need for all these computations, as the group is finite by Lemma~\ref{lemma:pitheta2_onesymmetryallsymmetry} and Thm.~\ref{thm:pitheta2:groupfinite}, and we can directly compute all polyharmonic functions.
If $|p|>1$, the singularity of $\Upsilon_3$ closest to $0$ (and thus the one determining asymptotic behaviour) is at $x=1$.  In the case $|p|<1$, however, $\Upsilon$ will have exponential growth with base $\frac{1}{p}$. We will now compute the exact shape of the resulting terms in this case. 

\begin{lemma}\label{generalization:lemma2}
Let $d$ be the radius of the circle $X[y_1,1]$, $c$ its center and $p$ its left intersection with the real axis. We then have 
\begin{align}\label{generalization:eq:lemma2}
d=\frac{(1-x_1)(1-x_4)}{2-x_1-x_4},
\end{align}
or, equivalently 
\begin{align}
    p-c(x_1+x_4)+x_1x_4=0.
\end{align}

\end{lemma}
\begin{proof}
One can check that, if $\pi/\theta=2$ and we have zero drift, then \begin{align}
    p=X(y_4)=X(y_1)=-\frac{\tilde{b}(y_1)}{2\tilde{a}(y_1)}=\frac{-1 + 2 p_{-1,0}}{1 + 4 p_{1,-1} - 4 p_{-1,-1} - 2 p_{-1,0}}.
\end{align}
Similarly, one sees that \begin{align}
    x_1x_4&=\frac{4 p_{-1,1} p_{-1,-1} - p_{-1,0}^2}{4 p_{-1,1} p_{1,-1} + 4 p_{1,-1} (p_{-1,-1} + p_{-1,0}) - (2 p_{-1,-1} + p_{-1,0})^2},\\
    x_1+x_4&=\frac{2 ((-1 + p_{-1,0}) p_{-1,0} + p_{-1,-1} (-1 + 2 p_{-1,0}) + p_{-1,1} (-1 + 4 p_{-1,-1} + 2 p_{-1,0}))}{
4 p_{-1,1} p_{1,-1} + 4 p_{1,-1} (p_{1,-1} + p_{-1,0}) - (2 p_{-1,-1} + p_{-1,0})^2}.
\end{align}
Putting the above together, (\ref{generalization:eq:lemma2}) follows immediately.
\end{proof}
The above Lemma~\ref{generalization:lemma2} now turns to be very useful for us: the term $L_3(x)$, which we want to integrate, contains a root of the form $\sqrt{-(x-x_1)(x-x_4)}$. We integrate along the contour $\Gamma$, thus we substitute $x\mapsto c+d e^{iu}$. It turns out that we now have, for any two constants $a,b$, \begin{align}
    -x^2-ax-b=-1 - b + a (-1 + d) + 2 d - 
 d e^{iu} (2 + a - 2 d + 2 d \cos{u}).
\end{align}
However, for $a=-(x_1+x_4)$ and $b=x_1x_4$, Lemma~\ref{generalization:lemma2} tells us that the constant term vanishes, i.e. we have \begin{align}
\label{generalization:eq5}
    -(x-x_1)(x-x_4)=-de^{iu}(2-x_1-x_4-2d+2d\cos{u}).
\end{align}
This will allow us to rewrite the square root, since the second factor is strictly real.\\
There are now two cases to consider: $x_4<p$ and $x_4>p$. As the computations are very similar, we will now only look at the former.\\
If $x_4<p$, then we have $-(p-x_1)(p-x_4)>0$, which means that we select the branch cut of the root to be along the positive real axis. Next, we will determine the sign of the real factor under the root. 

\begin{lemma}\label{generalization:lemma3}
    If $x_4<p$, then we have for any $u\in[0,2\pi]$ \begin{align}
        \label{generalization:eq:lemma3}
        2-x_1-x_4+2d(-1+\cos{u})\geq 0.
    \end{align}
\end{lemma}

\begin{proof}
First, notice that \begin{align}
\label{generalization:eq6}
    2-x_1-x_4+2d(-1+\cos{u})\geq (1-x_1)+(1-x_4)-4d.
\end{align}
Due to Lemma~\ref{generalization:lemma2}, we know that \begin{align}
    d=\frac{(1-x_1)(1-x_4)}{(1-x_1)+(1-x_4)}.
\end{align}
The right side of (\ref{generalization:eq6}) can therefore be written as \begin{align}
    a+b-4\frac{ab}{a+b},
\end{align}
for $a=1-x_1,b=1-x_4$. Since $x_4<-1<x_1$, we know that $a>b>0$, utilizing homogeneicity it therefore suffices to study the function \begin{align}
    f(a)=1+a-4\frac{a}{1+a}=\frac{(1-a)^2}{1+a}.
\end{align}
As $f(a)$ does not have any zeros for $a>1$, (\ref{generalization:eq:lemma3}) holds.
\end{proof}

Due to our choice of branch cut, we have 
\begin{small}\begin{align}
\sqrt{-(x-x_1)(x-x_4)}&=\sqrt{-de^{iu}(2-x_1-x_4-2d+2d\cos{u})}\\
&=\left|\sqrt{d(2-x_1-x_4-2d+2d\cos{u})}\right|\sqrt{e^{i(u+\pi)}}\\
\label{generalization:eq7}
&=\begin{cases}
    \left|\sqrt{d(2-x_1-x_4-2d+2d\cos{u})}\right|\cdot e^{i(u+\pi)/2},\quad\text{if}\quad u\in[0,\pi),\\
    \left|\sqrt{d(2-x_1-x_4-2d+2d\cos{u})}\right|\cdot e^{i(u-\pi)/2},\quad\text{if} \quad u\in[\pi,2\pi).
\end{cases}
   \end{align}
   \end{small}\noindent

If we now rewrite the integral \begin{align}
    \int_\Gamma \frac{L_3(x)}{x-t}\mathrm{d}x&=d\int_0^{2\pi}\frac{L_3(c+de^{iu})e^{iu}}{c+de^{iu}-t}\mathrm{d}u
\end{align}
using (\ref{generalization:eq7}), it turns out that the only term generating arbitrarily high powers of $t$ is 
\begin{small}\begin{multline}
    \sqrt{2b}(-1 + 2 c - 
    t) \sqrt{(a - b) (2 a (-1 + c) (c - t) + 
     b (1 + 2 (-1 + c) c - 2 c t + t^2))}\\
     \cdot\operatorname{ArcTan}\left(\frac{\sqrt{a-b} (-1 + t)}{\sqrt{
   4 a (-1 + c) (c - t) + 2 b (1 + 2 (-1 + c) c - 2 c t + t^2)}}\right),
\end{multline}
\end{small}
where $a=d(2-x_1-x_4+2d)$, $b=2d^2$. Up to a multiplicative constant, we can write this as \begin{align}\label{generalization:eq8}
    T(t):=(-1+2c-t)\sqrt{(t-x_1)(t-x_4)}\operatorname{ArcTan}\left(\frac{\sqrt{d(2-x_1-x_4-4d)}(t-1)}{2d\sqrt{(t-x_1)(t-x_4)}}\right).
\end{align}
From here on, it remains to do a standard computation, using the theory developed e.g. in \cite{Flajolet}. We know that $\operatorname{ArcTan}{x}$ has its singularities at $x=\pm i$; one can check that this is the case for $t=-p$. The singularities due to the $\sqrt{(t-x_1)(t-x_4)}$ cancel, since the series representation of $x\operatorname{ArcTan}{x}$ contains only even powers of $x$. Therefore, all we need to do is to consider the behaviour of the right-hand side of (\ref{generalization:eq8}) near $t=-p$.\\
One can check that, after considering $\tilde{T}(t):=T(pt)$ in order to shift the pole to $t=1$, we have an asymptotic expansion of the form 
\begin{align}
    T(t)=c_1+c_2\log{|1-t|}+\mathcal{O}(1-t)
\end{align}
for some constants $c_1,c_2$, and thus, according to e.g. \cite[Thm.~VI.3]{Flajolet}, we have \begin{align}
    [t^n]T(t)=\mathcal{O}\left[\frac{1}{p^n}\left(\frac{\log{n}}{n}\right)\right].
\end{align}
In particular, considering the $\log$-terms, we know that $T(t)$, and hence also $\Upsilon(x)$ and the resulting decoupling function, cannot be algebraic.

\subsection{An example with infinite group}\label{sec:pitheta2_infinite_example}
We will now look at how the computations above work in a concrete example with zero drift, $\pi/\theta=2$ and infinite group. As there are no models as famous as say the simple walk or the tandem walk which belong to this group, we pick the same model as in Example~\ref{sec:ansatz_infiniteexample}, which was defined by \begin{align}
 p_{1,0}=p_{0,1}=0,\\
 p_{1,1}=1/4,\\
 p_{1,-1}=p_{0,-1}=p_{-1,0}=p_{-1,1}=1/6,\\
 p_{-1,-1}=1/12.
\end{align}
Here, $\Gamma$ is the circle defined by $|z-1/6|^2=\left(\frac{5}{6}\right)^2$, and therefore complex conjugation on $\Gamma$ corresponds to the M\"obius transform $r: z\mapsto \left(\frac{5}{6}\right)^2\frac{1}{z-1/6}+\frac{1}{6}$. After some computation, one checks that $L_1(x)$ and $L_2(x)$ as defined in (\ref{eq:pitheta2_infinite_L1}) and (\ref{eq:pitheta2_infinite_L2}) now take the form
\begin{align}
    L_1(x)&=\frac{8(2+3x)(7+6x(1+2x))}{125(1-x)^3}\\
    L_2(x)&=\frac{8(2+3x)(x-1)\sqrt{(6x-1)(3+x)}}{125(1-x)^3}.
\end{align}
For a decoupling function of $L_1(x)$, one obtains \begin{align}
  \Upsilon_1(x)=\frac{4}{(x-1)^3}+\frac{96}{25(x-1)}.
\end{align}
For a decoupling function $\Upsilon_2(x)$ of $L_2(x)$, we proceed as in Section~\ref{sec:pitheta2_infinite_decouplingnonrational}. We start by letting $L_3(x)=\frac{-8(2+3x)\sqrt{-1-18x-6x^2}}{25(1-6x)}$, and compute in a first step a decoupling function $\Upsilon_3(x)$ of $L_3(x)$. Using the same integration trick as above, we finally arrive at \begin{align}
    2\pi i\Upsilon_3(t)+2\pi i\Upsilon_3\left(1/6\right)=\int_\Gamma \frac{L_3(x)}{x-t}\mathrm{d}x.
\end{align}
We remember that we can always add a constant to a decoupling function without changing the decoupling property, so we can assume that $\Upsilon_3(1/6)=0$ and thus have 
\begin{align}
    \Upsilon_3(t)=\frac{1}{2\pi i} \int_\Gamma\frac{L_3(x)}{x-t}\mathrm{d}x.
\end{align}
Combining everything until now, we have \begin{align}
    \Upsilon(x)=\frac{\frac{1}{2\pi i}\int_\Gamma\frac{L_3(t)}{t-x}\mathrm{d}t}{(x-1)(r(x)-1)}+\frac{4}{(x-1)^3}-\frac{96}{25(x-1)}.
\end{align} 
Here, the interesting part is clearly the contour integral. Thus, we want to know the asymptotics of the coefficients of \begin{align}
    \Upsilon_3(t)=\frac{1}{2\pi i}\int_\Gamma\frac{(x+2/3)\sqrt{-(1+18x+6x^2)}}{(x-1/6)(x-t)}.
\end{align}
By some computation, one finds that\begin{align}
    \left[t^n\right]\Upsilon_3(t)\sim\left[t^n\right](2+3t)\sqrt{1/3+6t+2t^2}\arctan\left(\frac{t-1}{\sqrt{1/3+6t+2t^2}}\right).
\end{align} 
The singularity closest to $0$ here is at $-\frac{2}{3}$, and utilizing that $\Upsilon_3\left(-\frac{2}{3}z\right)=\mathcal{O}\left(\log\frac{1}{1-z}\right)$, we can apply \cite[Thm.~VI.3]{Flajolet} and have 
\begin{align}
    \left[z^n\right]\Upsilon_3(z)=\left(-\frac{3}{2}\right)^n\left[z^n\right]\Upsilon_3\left(-\frac{2}{3}z\right)=\mathcal{O}\left(\left(-\frac{3}{2}\right)^n\frac{\log n}{n}\right).
\end{align}
As mentioned, we can therefore deduce that $\Upsilon_3(t)$, and therefore also $\Upsilon(t)$, is not algebraic.
\vskip\baselineskip
Accordingly, in case of an infinite group, the existence of a decoupling function as postulated by general theory about boundary value problems does not appear to be very useful in terms of actual computations. If one were to drop the condition that $\pi/\theta=2$ on top of this, then the calculations would yet again get much more complicated, as we were heavily relying on the fact that we can describe complex conjugation on the contour $\Gamma$ via a simple rational transformation, which is only due to the fact that $\Gamma$ is a circle. Seeing as the general method to compute polyharmonic functions as presented in Section~\ref{sec:generalSolution} do not correspond to a continuous solution, it would still be interesting to see if there is a natural way to construct polyharmonic functions in the infinite group case.

\section{Outlook/Open Questions}\label{sec:Outlook}
There are a number of open questions regarding discrete polyharmonic functions. \begin{itemize}
    \item It would be interesting to know if one can in general find asymptotics of the form as in Lemma~\ref{lemma:asymptotics_PHF} for the kind of counting problem stated in the introduction. For the one-dimensional case, this was recently shown in \cite{WachtelPhD}. It might be that there are wider classes of problems where discrete polyharmonic functions appear in the asymptotics as well.
\item Directly tied to this, it would be reasonable to expect that not all discrete polyharmonic functions would appear in a combinatorial context, but only those with a somewhat simple structure. For harmonic functions, for instance, in most combinatorial applications it is only the -- in the zero drift context unique up to multiples -- positive harmonic function which relevant. However, even for harmonic functions, while it is conjectured in \cite{Hung} that the positive harmonic function is always given by $H_1^1(x,y)$, this is to the author's knowledge not yet proven in general. It is also not clear at all whether this notion of positivity, or of combinatorial relevance, can be extended to biharmonic, or polyharmonic functions. Maybe this problem could be tackled by limit properties; as one would for many problems expect the discrete solution, together with the asymptotics, to converge towards to continuous limit. However, this could only serve as a starting point, seeing as a discrete polyharmonic function is not uniquely defined by its scaling limit.
\item  Another natural question would be whether there is a better, or more intuitive way to compute discrete polyharmonic functions in the infinite group case. While the method presented in Section~\ref{sec:generalSolution} works, one could argue that any method which does not give the most simple, rational basis if it exists is probably not ideal. It seems that it would be very difficult, for example, to express the polyharmonic functions (\ref{eq:as_SW_1}), (\ref{eq:as_SW_2}) via those obtained by the general method. Also, it is a bit inconvenient that the scaling limit yields only a formal solution of the continuous equation, of which we cannot perform an inverse Laplace transform.
\item 
The models considered in this article all have small steps and zero drift, but it might be possible to extend them to or find similar theorems for a more general settings, or to consider singular models. For harmonic functions where the condition on the steps being small is relaxed, see for example \cite{Hung2}. In case of a non-zero drift, one frequently obtains so-called $t$-harmonic functions, where the $f(i,j)$ on the right-hand side of (\ref{eq:intro_1}) is multiplied by some constant factor $t$, see e.g. \cite{tharmonic}. It might therefore be of interest to study $t$-polyharmonic functions as well. 
\item Compared to the quarter plane, the study of walks in the octant has turned out to be a lot more complicated, seeing as the resulting functional equations have one more boundary term. It might therefore be ambitious, but ultimately still of interest to try and figure out if one can extend the results of this article to a model with more than $2$ dimensions.

\end{itemize}

\printbibliography[heading=subbibintoc]

\end{document}